\documentclass{amsart}
\usepackage{amssymb,amsthm,amsmath}
\usepackage[french]{babel}
\usepackage[margin=3.0cm]{geometry}
\usepackage{enumerate}
\usepackage[T1]{fontenc}

\newtheorem{theorem}{Th\'eor\`eme}
\newtheorem{main}{Th\'eor\`eme}

\newtheorem{definition}{D\'efinition}
\newtheorem{proposition}{Proposition}
\newtheorem{lemma}{Lemme}
\newtheorem{corollary}{Corollaire}
\newtheorem{exam}{Exemple}

\newtheorem{exams}{Examples}

\newtheorem{rmk}{Remarque}
\newenvironment{remark}{\begin{rmk}\rm}{\end{rmk}}
\newtheorem{notat}{Notation}

\title[Feuilletages holomorphes de codimension 1: une \'etude locale]{Feuilletages holomorphes de codimension 1: une \'etude locale dans le cas dicritique}
\author{D. Cerveau, A. Lins Neto \& M. Ravara-Vago}

\address{Dominique Cerveau, Institut Universitaire de France et IRMAR, Universit\'e de Rennes 1, Campus de Bealieu, 35042 Rennes Cedex France.}\email{dominique.cerveau@univ-rennes1.fr}

\address{Alcides Lins Neto, IMPA, Estrada Dona Castorina 110, Rio de Janeiro, Brasil.}\email{alcides@impa.br}

\address{Marianna Ravara-Vago, IRMAR, Universit\'e de Rennes 1, Campus de Beaulieu, 35402 Rennes Cedex France.}\email{ravaravago@gmail.com}

\keywords{Singular holomorphic foliations, finite determinacy, unfolding, Liouvillian integration}

\date{Avril 2014}
\begin{document}

\begin{abstract}
Nous d\'ecrivons les singularit\'es de feuilletages holomorphes {\em dicritiques} de petite multiplicit\'e en dimension $3$. En particulier nous relions l'existence de d\'eformations et de d\'eploiements non triviaux \`a des probl\`emes d'int\'egrabilit\'e liouvillienne.
\end{abstract}

\maketitle

\tableofcontents

\section{Introduction}

On se propose dans cet article de donner la description de certains germes de feuilletages holomorphes de codimension un \`a l'origine de ${\mathbb C}^n$. Un tel feuilletage ${\mathcal F}$ est associ\'e \`a la donn\'e d'une $1-$forme holomorphe $\omega \in \Omega^1({\mathbb C}^n,0)$, d\'efinie \`a unit\'e multiplicative pr\`es, satisfaisant la condition d'int\'egrabilit\'e de Frobenius $\omega \wedge d\omega = 0$ et dont le lieu singulier ${\rm Sing}\ \omega =$ Z\'eros$(\omega)$ est de codimension sup\'erieure ou \'egale \`a deux. On note ${\mathcal F} = {\mathcal F}_\omega$ et ${\rm Sing}\ {\mathcal F} = {\rm Sing}\ \omega$. La condition d'int\'egrabilit\'e est non lin\'eaire en les coefficients de $\omega$, ce qui rend difficile la construction d'exemples par des proc\'ed\'es calculatoires. Il y a essentiellement deux mani\`eres simples pour construire de tels exemples. La premi\`ere consiste \`a se donner un germe $\omega_0 \in \Omega^1({\mathbb C}^2,0)$ en dimension $2$ (dans ce cas la condition d'int\'egrabilit\'e est triviale), une application holomorphe $F:{\mathbb C}^n,0 \rightarrow {\mathbb C}^2,0$ dominante et \`a consid\'erer $\omega = F^*\omega_0$ qui est automatiquement int\'egrable; les feuilles du feuilletage ${\mathcal F}_\omega$ sont alors fibr\'ees par les niveaux de $F$. Un tel feuilletage sera dit de type {\em pull-back}; on peut d'ailleurs g\'en\'eraliser cette proc\'edure en consid\'erant un feuilletage ${\mathcal F}_0$ d'une surface $S_0$, \'eventuellement singuli\`ere, et une application m\'eromorphe $F: {\mathbb C}^n \dashrightarrow S_0$ (typiquement la projection naturelle ${\mathbb C}^3 \dashrightarrow {\mathbb P}^2_{\mathbb C}$) et en rappelant ${\mathcal F}_0$ par $F$. L'autre mani\`ere est de consid\'erer les feuilletages ${\mathcal F}_\theta$ associ\'es aux $1-$formes m\'eromorphes $\theta$ ferm\'ees: $d\theta = 0$. Comme nous le rappellerons une telle forme s'\'ecrit: $$\theta = \sum \lambda_i \frac{df_i}{f_i} + dh $$ o\`u les $f_i$ sont holomorphes, $h$ m\'eromorphe et les {\em r\'esidus} $\lambda_i$ des nombres complexes. Essentiellement les feuilles de ${\mathcal F}_\theta$ sont les niveaux de la fonction multivalu\'ee $\sum \lambda_i {\rm log} f_i + h$. Ceci n'est pas sans rappeler la conjecture, concernant cette fois les feuilletages globaux, suivante:
\vskip0.2cm
\noindent {\bf Conjecture.} (Brunella, Lins Neto {\em et alt} \cite{Croc}) {\em Soit ${\mathcal F}$ un feuilletage holomorphe de codimension un sur une vari\'et\'e projective $X$. Alors ou bien ${\mathcal F}$ est transversalement projectif (sur un ouvert de Zariski) ou bien il existe $F:X \dashrightarrow S$ une application rationnelle vers une surface $S$ et un feuilletage ${\mathcal G}$ de S tels que ${\mathcal F} = F^*{\mathcal G}$.}
\vskip0.2cm

Pour \'eclairer cette conjecture disons que parmi les feuilletages transversalement projectifs il y a les feuilletages donn\'es par une $1-$forme ferm\'ee m\'eromorphe ou les feuilletages transverses \`a une fibration rationnelle (sur un ouvert de Zariski fibr\'e). Ils sont donn\'es par une $1-$forme rationnelle $\omega_0$ pour laquelle existe deux autres $1-$formes rationnelles $\omega_1$, $\omega_2$ formant un $SL(2,{\mathbb C})$ triplet \cite{Croc,LJVT}: 
$$d\omega_0 = \omega_0 \wedge \omega_1,\quad d\omega_1 = \omega_0 \wedge \omega_2,\quad d\omega_2 = \omega_1 \wedge \omega_2. $$

Soit $\omega \in \Omega^1({\mathbb C}^n,0)$ un germe de $1-$forme int\'egrable, $\omega = \sum a_idz_i$; si l'ensemble singulier \break ${\rm Sing}\ \omega:=\{a_1=\cdots=a_n=0\}$ est suffisament petit, ${\rm Cod\ Sing}\ \omega \geq 3$, le Th\'eor\`eme de Frobenius de B. Malgrange assure l'existence d'une int\'egrale premi\`ere non constante $f \in {\mathcal O}({\mathbb C}^n,0)$: $\omega = gdf$, $g \in {\mathcal O}^*({\mathbb C}^n,0)$; ici les feuilles sont les niveaux de $f$. On pourrait penser que cel\`a d\'ecrit la situation g\'en\'erique (pour la topologie de Krull), mais il n'en est rien: on sait en effet depuis Kupka et Reeb que la condition $d\omega(0)\neq0$ implique que le lieu singulier est lisse de codimension deux et cette propri\'et\'e est \'evidemment stable. Nous pr\'ecisons dans le texte ce ph\'enom\`ene bien connu et nous pr\'esentons la classification des $1-$formes int\'egrables dont le $1-$jet est non trivial suivant les travaux de \cite{Ku,Cer-Ma,L}. Une grande partie du travail que nous proposons repose sur l'id\'ee na\"ive suivante; si l'on proc\`ede au d\'eveloppement de Taylor de $\omega$: $$\omega = \omega_\nu + \omega_{\nu+1} + \cdots + \omega_k + \cdots $$ o\`u chaque $\omega_k$ est une $1-$forme homog\`ene de degr\'e $k$, i.e. \`a coefficients polyn\^omes homog\`enes de degr\'e $k$, et $\omega_\nu$ est la {\em partie homog\`ene de plus bas degr\'e non nulle}, alors la partie initiale ${\rm In}(\omega) = \omega_\nu$ de $\omega$ est int\'egrable. Mieux, il y a une homotopie $$ \omega_t = \omega_\nu + t\omega_{\nu+1} + \cdots + t^k\omega_{\nu+k} + \cdots, \quad t \in {\mathbb C} $$ reliant $\omega = \omega_{t=1}$ \`a sa partie initiale $\omega_{t=0} = {\rm In}(\omega)$. On peut donc esp\'erer, modulo des conditions raisonnables sur $\omega_\nu$, que la $1-$forme $\omega$, que l'on voit donc comme une d\'eformation de $\omega_\nu$, va conserver certaines propri\'et\'es de sa partie initiale ${\rm In}(\omega)$. Pour pr\'esenter nos r\'esultats nous avons besoin de quelques notations et rappels. On d\'esigne par $R$ le champ {\em radial}, $R = \sum z_i \frac{\partial}{\partial z_i}$. Une $1-$forme homog\`ene int\'egrable $\omega_\nu$ est dite {\em dicritique} si le polyn\^ome $P_{\nu+1} = i_R\omega_\nu$ est identiquement nul; elle est {\em non dicritique} sinon. Une forme dicritique induit un feuilletage sur l'espace projectif ${\mathbb P}^{n-1}_{\mathbb C}$, tandis que le feuilletage homog\`ene associ\'e \`a une forme non dicritique est d\'efini par une $1-$forme ferm\'ee rationnelle: en effet ${\displaystyle \frac{\omega_\nu}{P_{\nu+1}}}$ est ferm\'ee. Modulo des conditions g\'en\'eriques portant sur ${\omega}_\nu$, ces propri\'et\'es sont gard\'ees en m\'emoire par les $\omega$ telles que ${\rm In}(\omega) = \omega_\nu$. Ainsi, dans le cas non dicritique, si $[P_{\nu+1} = 0]\subset {\mathbb P}^{n-1}_{\mathbb C}$ est r\'eduit et \`a croisement normaux, alors ${\mathcal F}_\omega$ est encore d\'efini par une $1-$forme ferm\'ee m\'eromorphe d\`es que les r\'esidus $\lambda_i$ de $\omega_\nu = {\rm In}(\omega)$ satisfont une condition g\'en\'erique (satisfaite sur un ouvert dense); de m\^eme si $[P_{\nu+1} = 0]$ est irr\'eductible de degr\'e $p^s$, avec $p$ premier et $n\geq3$ \cite{Cer-L}: en fait dans ce cas $\omega$ poss\`ede une int\'egrale premi\`ere holomorphe non constante.

Dans le cas dicritique et en dimension $3$, on d\'emontre dans \cite{Cam-Alc} que si $d\omega_\nu = d({\rm In}(\omega))$ est \`a singularit\'e isol\'ee et $\nu \geq 3$, alors les feuilletages ${\mathcal F}_{\omega_\nu}$ et ${\mathcal F}_\omega$ sont holomorphiquement conjugu\'es, ce que l'on peut interpr\'eter comme un r\'esultat de {\em d\'etermination finie} ou de stabilit\'e (pour la topologie de Krull).

Nous nous int\'eressons dans cet article \`a des feuilletages ${\mathcal F}_\omega$ dont la partie homog\`ene $\omega_\nu = {\rm In}(\omega)$ ne satisfait plus les conditions pr\'ec\'edentes. Si $\omega_\nu$ est une $1-$forme homog\`ene int\'egrable dicritique, on note $[{\mathcal F}_{\omega_\nu}]$ le feuilletage de ${\mathbb P}^{n-1}_{\mathbb C}$ associ\'e; si ${\rm Cod\ Sing}\ \omega_\nu \geq 2$, c'est un feuilletage de degr\'e $\nu-1$. Rappelons qu'un point {\em central} (ou une singularit\'e de {\em type centre}) pour un feuilletage ${\mathcal G}$ du plan est un point singulier $m$ en lequel ${\mathcal G}$ poss\`ede une int\'egrale premi\`ere de Morse; toujours en dimension $2$, un point singulier $m$ est dit {\em nilpotent} si ${\mathcal G}_{,m}$ est d\'efini par un germe de $1-$forme dont la partie initiale est de type $xdx$.

Les r\'esultats qui suivent sont propres \`a la dimension $3$, premi\`ere dimension ou la condition d'int\'egrabilit\'e est non triviale.

\begin{main}\label{main} Soit $\omega \in \Omega^1({\mathbb C}^3,0)$ holomorphe int\'egrable. On suppose que la partie initiale ${\rm In}(\omega) = \omega_\nu$ est dicritique et satisfait ${\rm Cod\ Sing}\ \omega_\nu \geq 2$. Alors:
\begin{enumerate}
\item Si $\nu=1$, ${\mathcal F}_\omega$ est holomorphiquement conjugu\'e \`a ${\mathcal F}_{\omega_1}$, $\omega_1 = z_2dz_1 - z_1dz_2$; en particulier ${\mathcal F}_\omega$ poss\`ede int\'egrale premi\`ere m\'eromorphe ${\displaystyle \frac{Z_1}{Z_2}}$ ($Z_i$ coordonn\'ees).
\item Si $\nu=2$, ${\mathcal F}_\omega$ est donn\'e par une $1-$forme ferm\'ee m\'eromorphe.
\item Si $\nu \geq 3$ et $[{\mathcal F}_{\omega_\nu}]$ n'a ni singularit\'e nilpotente, ni centre, alors ${\mathcal F}_\omega$ est conjugu\'e \`a ${\mathcal F}_{\omega_\nu}$.
\item Si $\nu = 3$ et $[{\mathcal F}_{\omega_3}]$ a une singularit\'e de type centre, alors ${\mathcal F}_\omega$ est d\'efini par une $1-$forme ferm\'ee m\'eromorphe.
\item Si $\nu = 3$ et $[{\mathcal F}_{\omega_3}]$ a une singularit\'e nilpotente alors, modulo une condition de non r\'esonnance portant sur $\omega_3$, ${\mathcal F}_\omega$ est d\'efini par une $1-$forme ferm\'ee m\'eromorphe ou bien ${\mathcal F}_\omega$ est conjugu\'e \`a ${\mathcal F}_{\omega_3}$.
\item Si $\nu=3$ et ${\mathcal F}_{\omega_3}$ a une singularit\'e quadratique ($1-$jet nul), alors ${\mathcal F}_\omega$ est transversalement affine, i.e. $d\omega = \omega \wedge \omega_1$, avec $\omega_1$ m\'eromorphe ferm\'ee.
\end{enumerate}
\end{main}

Dans le Th\'eor\`eme \ref{main}, nous avons rassembl\'e divers \'enonc\'es qui apparaissent dans le texte. Par souci de coh\'erence nous avons ins\'er\'e dans le th\'eor\`eme des r\'esultats bien connus: c'est le cas des points 1 et 2. Le point 3 contient le r\'esultat de Camacho-Lins Neto \cite{Cam-Alc}. Le point 4 fait intervenir un r\'esultat d'int\'egration Liouvillienne remarquable d\'emontr\'e en 1908 par Henri Dulac \cite{D}. Le point 5 n\'ecessite une \'etude fine des d\'eploiements des feuilletages du plan \`a singularit\'e nilpotente. Cette \'etude est rendue possible par l'utilisation du Th\'eor\`eme de Pr\'eparation de F. Loray \cite{L}. A titre d'exemple, on d\'emontre que si $\omega_0 \in \Omega^1({\mathbb C}^2,0)$ est \`a singularit\'e nilpotente et \`a nombre de Milnor $\mu(\omega_0;0)$ de type $p-1$ avec $p$ premier, alors tout d\'eploiement ${\mathcal F}_\omega$ de ${\mathcal F}_{\omega_0}$ \`a singularit\'e nilpotente $({\rm In}(\omega) = {\rm In}(\omega_0))$ est trivial ou bien ${\mathcal F}_{\omega_0}$ (et ${\mathcal F}_\omega$) poss\`ede une int\'egrale premi\`ere holomorphe non constante. Rappelons que si $\omega_0 = adx + bdy$ alors $\mu(\omega_0;0)$ est la multiplicit\'e d'intersection de l'id\'eal $(a,b)$: $$\mu(\omega_0;0) = {\rm dim} \frac{{\mathcal O}({\mathbb C}^2,0)}{(a,b)}. $$

Nous donnons des exemples de d\'eploiements non triviaux pour les feuilletages ${\mathcal F}_{\omega_3}$ associ\'es \`a des $1-$formes ferm\'ees m\'eromorphes. En particulier on rencontre dans cette description des feuilletages ${\mathcal F}_\omega$, $\omega \in \Omega^1({\mathbb C}^3,0)$, qui apr\`es r\'eduction des singularit\'es sont compl\'etement r\'eguliers.

\vskip0.3cm
\noindent {\em Remerciements.} Nous remercions F. Loray, C. Rousseau et F. Touzet pour des \'echanges fructueux. Le deuxi\`eme auteur a b\'en\'efici\'e d'un financement IUF et CNPq pour un s\'ejour \`a Rennes, le troisi\`eme est financ\'e par un projet PDE/CsF du CNPq pour un s\'ejour de deux ans \`a Rennes.

\section{Notations, d\'efinitions et rappels}

On note $\mathcal{O}(\mathbb{C}^n,0)$ l'anneau des germes de fonctions holomorphes \`a l'origine de $\mathbb{C}^n$; on d\'esigne par $\Omega^k(\mathbb{C}^n,0)$ le $\mathcal{O}(\mathbb{C}^n,0)-$module des germes de $k-$formes holomorphes \`a l'origine de $\mathbb{C}^n$. Bien s\^ur $\Omega^0(\mathbb{C}^n,0) = \mathcal{O}(\mathbb{C}^n,0)$; le $\mathcal{O}(\mathbb{C}^n,0)-$module des champs de vecteurs holomorphes se note $\Theta(\mathbb{C}^n,0)$. Plus g\'en\'eralement si $M$ est une vari\'et\'e complexe on note $\mathcal{O}(M)$, $\Omega^k(M)$ et $\Theta(M)$ les faiseaux de fonctions, $k-$formes et champs. Si $\alpha\in\Omega^k$ et $X \in \Theta$ on note $i_X \alpha$ le produit int\'erieur de $\alpha$ par $X$ et par $L_X \alpha = i_X d\alpha + d(i_X\alpha)$ la d\'eriv\'ee de Lie de $\alpha$ suivant $X$.

Un germe de feuilletage holomorphe $\mathcal{F} = \mathcal{F}_\omega$ de codimension 1 \`a l'origine de $\mathbb{C}^n$ est d\'efini par la donn\'ee d'un germe de $1-$forme $\omega \in \Omega^1(\mathbb{C}^n,0)$ satisfaisant les deux conditions:
\begin{list}{(i)}
 \item $\omega\wedge d\omega=0$ (int\'egrabilit\'e)
\end{list}
\begin{list}{(ii)}
 \item ${\rm Cod\ Sing}\ \omega \geq 2$
\end{list} 
\noindent o\`u ${\rm Sing}\ \omega$ d\'esigne l'ensemble singulier de $\omega$; en coordonn\'ees si:
$$\omega = \sum\limits_{i=1}^{n} a_i dx_i,\ a_i \in \mathcal{O}(\mathbb{C}^n,0) $$
alors ${\rm Sing}\ \omega = \{a_1 = \cdots = a_n = 0\}$.

Ajoutons que par d\'efinition deux $1-$formes $\omega$, $\omega'$ satisfaisant (i), (ii) et $\omega\wedge\omega'=0$ d\'efinissent le m\^eme feuilletage. Avec le notation ci-dessus le lieu singulier ${\rm Sing}\ \mathcal{F}$ de $\mathcal{F} = \mathcal{F}_\omega$ est ${\rm Sing}\ \omega$. Deux germes de feuilletages $\mathcal{F}_\omega = \mathcal{F}$ et $\mathcal{F}_{\omega'} = \mathcal{F}'$ sont conjugu\'es s'il existe $\phi \in \rm{Diff}(\mathbb{C}^n,0)$, le groupe des germes de diff\'eomorphismes de $\mathbb{C}^n,0$, tel que $\phi^* \omega \wedge \omega' = 0$. On note $\mathcal{F}'= \phi^{-1}\mathcal{F}$. 

On a le th\'eor\`eme de structure suivant:

\begin{theorem}{\em(Malgrange \cite{Mal})} Soit $\omega$ comme ci-dessus; si ${\rm Cod\ Sing}\ \omega \geq 3$ il existe une unit\'e $g \in \mathcal{O}^*(\mathbb{C}^n,0)$ et $f \in \mathcal{O}(\mathbb{C}^n,0)$ tels que $\omega = g df$.
\end{theorem}

Lorsque $\omega$ est non singuli\`ere, i.e. ${\rm Sing}\ \omega = \emptyset$, on retrouve le Th\'eor\`eme de Frobenius classique. Dans ce cas, $f$ est une submersion. 

Dans la th\'eorie des singularit\'es de feuilletage il y a un \'enonc\'e classique connu sous le nom de ph\'enom\`ene de Kupka-Reeb et qui est une application directe du Th\'eor\`eme de Darboux classifiant localement les $2-$formes ferm\'ees. En voici l'\'enonc\'e:

\begin{theorem}\cite{Ku,Can-Cer-Des}
Soit $\mathcal{F} = \mathcal{F}_\omega$ un germe de feuilletage \`a l'origine de $\mathbb{C}^n$ d\'efini par la $1-$forme int\'egrable $\omega$. On suppose que la diff\'erentielle $d\omega$ est non nulle \`a l'origine: $d\omega(0) \neq0$. Il existe des coordonn\'ees $(x_1,x_2,\ldots,x_n)$ telles que $\omega = A_1(x_1,x_2)dx_1 + A_2(x_1,x_2) dx_2$, $A_i \in \mathcal{O}(\mathbb{C}^2,0)$.
\end{theorem} 

Sous les hypoth\`eses du th\'eor\`eme, il existe une submersion $\varphi:\mathbb{C}^n,0 \rightarrow \mathbb{C}^2,0$ telle que $\mathcal{F} = \varphi^{-1}(\mathcal{F}_0)$ o\`u $\mathcal{F}_0$ est un feuilletage sur $\mathbb{C}^2,0$.

Si l'on souhaite d\'ecrire tous les feuilletages donn\'es par une $1-$forme $\omega$ dont la partie lin\'eaire est non triviale les deux \'enonc\'es pr\'ec\'edents s'av\`erent insuffisants. Ils doivent \^etre compl\'et\'es par deux autres; le premier se trouve dans \cite{Cer-Ma}, nous en donnons sa preuve dont l'esprit est important pour la compr\'ehension de la suite:

\begin{proposition}\label{prop3}
Soit $\mathcal{F}$ un germe de feuilletage holomorphe de codimension 1 \`a l'origine de $\mathbb{C}^n$ donn\'e par $\omega \in \Omega^1(\mathbb{C}^n,0)$, $n\geq3$. Soient $i: \mathbb{C}^2,0 \hookrightarrow \mathbb{C}^n,0$ une immersion et $\omega_0 = i^* \omega$. On suppose que le $1-$jet de $\omega_0$ est non nilpotent (i.e. non nul et non conjugu\'e \`a $x_1dx_1$). On a l'alternative:
\begin{enumerate}[$(1)$]
\item $\mathcal{F}$ ne d\'epend que de deux variables, i.e. il existe une submersion $j: \mathbb{C}^n,0 \rightarrow \mathbb{C}^2,0$, $j \circ i = \rm{id}_{\mathbb{C}^2}$, telle que $\mathcal{F} = j^{-1}(i^{-1}\mathcal{F})$.
\item $\mathcal{F}$ poss\`ede une int\'egrale premi\`ere holomorphe $f \in \mathcal{O}(\mathbb{C}^n,0)$ non constante, i.e. $\omega \wedge df = 0$. 
\end{enumerate}
\end{proposition}

\begin{proof} Soit $(x,y) = (x_1,x_2,y_3,\ldots,y_n)$ un syst\`eme de coordonn\'ees tel que $i$ s'identifie \`a l'inclusion $i(x) = (x,0)$. Si $d\omega_0(0) \neq 0$ on a bien s\^ur $d\omega(0) \neq 0$ et on se trouve en pr\'esence d'un ph\'enom\`ene de Kupka-Reeb; on est alors dans le cas $(1)$. Sinon, puisque le $1-$jet de $\omega_0$ est non nilpotent, on a \`a conjugaison lin\'eaire pr\`es:
$$j^1\omega_0 = x_2dx_1 + x_1dx_2. $$
On \'ecrit alors: $$\omega = A_1 dx_1 + A_2 dx_2 + B_3 dy_3 + \cdots + B_n dy_n  $$ o\`u $A_i$, $B_i \in \mathcal{O}(\mathbb{C}^n,0)$ avec $A_1 = x_2 + \cdots$, $A_2 = x_1 + \cdots$. Si le lieu singulier de $\omega$ est de codimension au moins trois, le Th\'eor\`eme de Malgrange dit que nous sommes dans l'\'eventualit\'e $(2)$. Comme l'ensemble lisse $\{A_1=A_2=0\}$ est de codimension $2$ et contient ${\rm Sing}\ \omega$, on constate, lorsque ${\rm Cod\ Sing}\ \omega = 2$, l'\'egalit\'e ${\rm Sing}\ \omega = \{A_1 = A_2 = 0\}$. En particulier chaque $B_i$ est dans l'id\'eal $<A_1,A_2>$: $$B_i = \alpha_i A_1 + \beta_i A_2,\quad \alpha_i,\beta_i \in \mathcal{O}(\mathbb{C}^n,0). $$ Par suite les champs de vecteurs $X_j = \frac{\partial}{\partial y_j} - \alpha_j \frac{\partial}{\partial x_1} - \beta_j \frac{\partial}{\partial x_2}$ annulent $\omega$ par produit int\'erieur et produisent une \emph{trivialisation} de $\mathcal{F}$: nous sommes dans le cas $(1)$.\end{proof}

La description des feuilletages donn\'es par une $1-$forme dont le $1-$jet est $x_1dx_1$ est pr\'esent\'ee par F. Loray dans \cite{L} comme cons\'equence de son Th\'eor\`eme de Pr\'eparation. En voici l'\'enonc\'e:

\begin{theorem}\label{loray}
Soit $\mathcal{F} = \mathcal{F}_\omega$ un germe de feuilletage de codimension 1 \`a l'origine de ${\mathbb C}^n$. On suppose que la partie lin\'eaire de $\omega$ est de type $x_1dx_1$. Alors \`a conjugaison et multiplication par unit\'e pr\`es $\omega$ est de la forme: 
\begin{equation}\label{formeloray} 
\omega = x_1dx_1 + (l_1(f) + x_1l_2(f)) df
\end{equation} o\`u $f = f(x_2,\ldots,x_n) \in \mathcal{O}(\mathbb{C}^{n-1},0)$, $l_1$ et $l_2$ sont dans $\mathcal{O}(\mathbb{C},0)$.
\end{theorem}

Une $1-$forme de type (\ref{formeloray}) sera appel\'ee {\em forme normale de Loray}. Comme dans le ph\'enom\`ene de Kupka-Reeb, le feuilletage $\mathcal{F}$ apparait comme image r\'eciproque (par l'application $(x_1,f)$) d'un feuilletage de $\mathbb{C}^2,0$ (le feuilletage d\'efini par $x_1dx_1 + (l_1(y_1) + x_1l_2(y_1))dy_1$). 

Soit $\omega$ un germe de $1-$forme int\'egrable \`a l'origine de $\mathbb{C}^n$; on note $\nu = \nu(\omega)$ la multiplicit\'e alg\'ebrique de $\omega$. Si on effectue le d\'eveloppement de Taylor de $\omega$:
$$\omega = \omega_\nu + \omega_{\nu+1} + \cdots + \omega_k + \cdots $$ o\`u chaque $\omega_k$ est une $1-$forme homog\`ene de degr\'e $k$, alors $\omega_\nu$ est une $1-$forme homog\`ene qui se trouve int\'egrable (condition (i)); par contre la condition (ii) pour $\omega$, i.e. ${\rm Cod\ Sing}\ \omega \geq 2$ n'implique pas qu'elle soit v\'erifi\'ee par $\omega_\nu$. 

\begin{definition}
Soit $\omega_\nu$ une $1-$forme homog\`ene int\'egrable sur $\mathbb{C}^n$, $\omega_\nu = \sum\limits_{i=1}^{n}A_idx_i$, $A_i$ homog\`ene de degr\'e $\nu$. On dit que $\omega_\nu$ est dicritique si $\sum\limits_{i=1}^{n}x_iA_i \equiv 0$ et non dicritique sinon.
\end{definition}

Consid\'erons le champ \emph{radial} $R = \sum\limits_{i=1}^{n}x_i \frac{\partial}{\partial x_i}$; $\omega_\nu$ est dicritique si et seulement si le polyn\^ome $P_{\nu+1} = i_R \omega_\nu$ est identiquement nul. Dans le cas non dicritique, $P_{\nu+1} \not\equiv0$, la $1-$forme $\displaystyle{\frac{\omega_\nu}{P_{\nu+1}}}$ est ferm\'ee; si la d\'ecomposition en facteurs irr\'eductibles de $P_{\nu+1}$ est:
$$P_{\nu+1} = P_1^{n_1+1}\cdots P_s^{n_s+1},\quad n_i \in \mathbb{N}, P_i\ \mbox{irr\'eductible}, $$
alors il existe des nombres $\lambda_i \in \mathbb{C}$ et un polyn\^ome $H$ de degr\'e $d^0H = \sum n_id^0P_i$ tel que (\cite{Cer-Ma}):
$$\frac{\omega_\nu}{P_{\nu+1}} = \sum \lambda_i \frac{dP_i}{P_i} + d \left( \frac{H}{P_1^{n_1} \cdots P_s^{n_s}} \right). $$
Cette $1-$forme poss\`ede \emph{l'int\'egrale premi\`ere multivalu\'ee} $\displaystyle{\sum\lambda_i {\rm log} P_i + \frac{H}{P_1^{n_1}\cdots P_s^{n_s}}}$. 

Lorsque $P_{\nu+1}$ est r\'eduit, i.e. $n_i = 0$, on peut supposer $H=0$ et:
$$\frac{\omega_\nu}{P_{\nu+1}} = \sum \lambda_i \frac{dP_i}{P_i}\quad \mbox{avec}\ \sum \lambda_i d^0P_i =1.$$
Dans ce cas on dit que $\omega_\nu$ (ou $\displaystyle{\frac{\omega_\nu}{P_{\nu+1}}}$) est {\em logarithmique} et que les $\lambda_i$ sont les {\em r\'esidus} de $\displaystyle{\frac{\omega_\nu}{P_{\nu+1}}}$. 

On dit que les r\'esidus satisfont la condition $\mathcal{P}$ si $\lambda_i \not\in\mathbb{Q}$ pour tout $i$ et l'un des $\lambda_i$ est ou bien non r\'eel ou bien un irrationnel mal approch\'e par les rationnels. 

Voici deux r\'esultats dus \`a D. Cerveau et F. Loray \cite{Cer-L} et D. Cerveau et J.-F. Mattei \cite{Cer-Ma} que nous concentrons dans un seul et qui donnent un aper\c cu raisonnable des feuilletages $\mathcal{F} = \mathcal{F}_\omega$ lorsque $\omega_\nu$ est non dicritique logarithmique.

\begin{theorem}
Soit $\omega = \omega_\nu + \cdots$ un germe de $1-$forme int\'egrable \`a l'origine de $\mathbb{C}^n$, $n\geq3$. On suppose que la partie homog\`ene de plus bas degr\'e $\omega_\nu$ est non dicritique et satisfait ${\rm Cod\ Sing}\ \omega_\nu \geq 2$. 
\begin{enumerate}[$(1)$]
\item Si $P_{\nu+1} = i_R\omega_\nu$ est irr\'eductible de degr\'e $\nu+1 = p^s$ avec $s\in \mathbb{N}$ et $p$ premier, alors $\omega$ poss\`ede une int\'egrale premi\`ere holomorphe non constante $f = P_{\nu+1}+\cdots$, $\omega\wedge df=0$.
\item Si $P_{\nu+1} = P_1 \cdots P_s$ est r\'eduit, $s\geq 2$, l'hypersurface $[P_{\nu+1}=0] \subset \mathbb{P}^{n-1}_{\mathbb{C}}$ est \`a croisement ordinaires et les r\'esidus $\lambda_i$ satisfont la condition $\mathcal{P}$, alors $\omega$ poss\`ede une int\'egrale premi\`ere $\sum \lambda_i {\rm log} f_i$, o\`u les $f_i = P_i + \cdots$ sont holomorphes.
\item Si $P_{\nu+1}$ est irr\'eductible et  $[P_{\nu+1}=0] \subset \mathbb{P}^{n-1}_{\mathbb{C}}$ est \`a croisements ordinaires alors $\omega$ poss\`ede int\'egrale premi\`ere holomorphe. 
\end{enumerate}
\end{theorem}

\begin{remark}
Dans les cas $(1)$, $(2)$ et $(3)$ le feuilletage $\mathcal{F}_\omega$ est d\'efini par une $1-$forme ferm\'ee tout comme sa \emph{partie homog\`ene} $\mathcal{F}_{\omega_\nu}$.
\end{remark}

\begin{remark}
Dans \cite{Cer-L} on construit des feuilletages $\mathcal{F}_\omega$ de $\mathbb{C}^3,0$ tels que la partie homog\`ene $\omega_\nu$ soit de type $dP_{\nu+1}$ avec $P_{\nu+1}$ irr\'eductible et qui ne poss\`edent pas d'int\'egrale premi\`ere holomorphe non constante. Ils ont la propri\'et\'e d'\^etre image r\'eciproque par une application holomorphe d'un feuilletage de $\mathbb{C}^2,0$.
\end{remark}

\begin{remark}
Dans \cite{Cer-Mo} on s'int\'eresse au cas o\`u $P_{\nu+1} =0$ est un croisement normal, i.e. un produit de coordonn\'ees, et les $\lambda_i$ des entiers positifs. On obtient dans ce cas l'alternative (non exclusive): ou bien $\mathcal{F}_\omega$ est d\'efini par une $1-$forme ferm\'ee (qui n'est pas en g\'en\'eral \`a p\^oles simples) ou bien $\mathcal{F}_\omega$ est image r\'ecriproque d'un feuilletage de $\mathbb{C}^2,0$.
\end{remark}

Int\'eressons nous maintenant dans le cas dicritique aux propri\'et\'es de $\omega_\nu$ suceptibles d'\^etre transmises \`a tout $\omega = \omega_\nu + \cdots$. Le premier r\'esultat en ce sens est du \`a C. Camacho et A. Lins Neto \cite{Cam-Alc}: 

\begin{theorem}\label{theo10}
\begin{enumerate}[$(1)$]
\item Soit $\omega = \omega_\nu + \cdots$ un germe de $1-$forme int\'egrable \`a l'origine de $\mathbb{C}^3$ dont le $\nu-$jet est la $1-$forme homog\`ene dicritique $\omega_\nu$ avec $\nu\geq 3$. Si la $2-$forme $d\omega_\nu$ ne s'annule qu'en 0, il existe un germe de diff\'eomorphisme $\phi: \mathbb{C}^3,0 \circlearrowleft$ tel que $\omega = \phi^*\omega_\nu$, i.e. $\omega$ et $\omega_\nu$ sont conjugu\'ees.
\item Soit $\omega \in \Omega^1(\mathbb{C}^n,0)$, $n\geq4$, un germe de $1-$forme int\'egrable pour lequel il existe une section $3-$plane $i:\mathbb{C}^3,0 \rightarrow \mathbb{C}^n,0$ telle que $i^*\omega$ satisfait le point $(1)$. Alors il existe une submersion $j:\mathbb{C}^n,0 \rightarrow \mathbb{C}^3,0$ telle que $\omega = j^*(i^*\omega)$. En particulier $\omega$ est conjugu\'e \`a la $1-$forme $(i^*\omega)_\nu \in \Omega^1(\mathbb{C}^3,0)$ (vue dans $\Omega^1(\mathbb{C}^n,0)$).
\end{enumerate}
\end{theorem}

Le Th\'eor\`eme \ref{theo10} classifie les feuilletages en dimension $3$ \`a partie homog\`ene \emph{g\'en\'erique} pour $\nu \geq 3$. Qu'en est-il pour $\nu \leq 2$? Il se trouve que si $\omega_\nu \in \Omega^1(\mathbb{C}^n,0)$ est int\'egrable homog\`ene dicritique de degr\'e $\nu \leq 2$, il existe un polyn\^ome homog\`ene $Q_{\nu+1}$ de degr\'e $\nu+1$ tel que $\displaystyle{\frac{\omega_\nu}{Q_{\nu+1}}}$ soit ferm\'ee \cite{Can-Cer-Des}. Par exemple si $\nu =1$, $\omega_1$ est lin\'eairement conjugu\'e \`a $x_1dx_2 - x_2dx_1$ et $Q_2 = x_1x_2$ convient. 

On d\'eduit la classification (g\'en\'erique lorsque) $\nu\leq2$ en utilisant le r\'esultat suivant \'enonc\'e dans \cite{Cer}: 

\begin{theorem}\label{theo11}
Soit $\omega = \omega_\nu +\cdots \in \Omega^1(\mathbb{C}^n,0)$ int\'egrable avec $\omega_\nu$ homog\`ene dicritique de degr\'e $\nu$ et satisfaisant ${\rm Cod\ Sing}\ \omega_\nu \geq2$. On suppose qu'il existe un polyn\^ome homog\`ene $Q$ tel que $\displaystyle{\frac{\omega_\nu}{Q}}$ soit ferm\'ee. Il existe alors $f \in \mathcal{O}(\mathbb{C}^n,0)$ tel que $\displaystyle{\frac{\omega}{f}}$ soit ferm\'ee.
\end{theorem}

\begin{remark}
De la m\^eme fa\c con que dans le cas homog\`ene, une $1-$forme ferm\'ee \emph{m\'eromorphe} $\eta$ s'\'ecrit:
$$\eta = \sum\limits_{i=1}^{s}\lambda_i \frac{df_i}{f_i} + d\left( \frac{H}{f_1^{n_1}\cdots f_s^{n_s}}\right) $$ o\`u les $f_i$ et $H$ sont holomorphes, $\lambda_i \in \mathbb{C}$ et $n_i \in \mathbb{N}$. La fonction multivalu\'ee $\displaystyle{\sum \lambda_i} {\rm log} f_i + \frac{H}{f_1^{n_1}\cdots f_s^{n_s}}$ est une int\'egrale premi\`ere du feuilletage correspondant.
\end{remark}

\begin{corollary}
Soit $\omega = \omega_\nu + \cdots \in \Omega^1(\mathbb{C}^n,0)$ int\'egrable avec $\omega_\nu$ homog\`ene dicritique, ${\rm Cod\ Sing}\ \omega_\nu \geq2$ et $\nu \leq2$. Il existe $f\in \mathcal{O}(\mathbb{C}^n,0)$ tel que $\displaystyle{\frac{\omega}{f}}$ soit ferm\'ee. 
\end{corollary}

\begin{remark}
Le Th\'eor\`eme \ref{theo11} s'\'etend stricto-sensu (avec une d\'emonstration diff\'erente) de la fa\c con suivante: s'il existe $Q$ rationelle homog\`ene telle que $\displaystyle{\frac{\omega_\nu}{Q}}$ soit ferm\'ee alors il existe $f$ m\'eromorphe telle que $\displaystyle{\frac{\omega}{f}}$ soit ferm\'ee.
\end{remark}

\begin{remark}
En dimension sup\'erieur \`a $4$ l'espace des formes int\'egrables homog\`enes dicritiques de degr\'e $\nu$ est une vari\'et\'e alg\'ebrique non irr\'eductible. ``Beaucoup'' de ses composantes sont associ\'es \`a des feuilletages d\'efinis par des $1-$formes ferm\'ees rationelles. C'est un cas typique  o\`u le Th\'eor\`eme \ref{theo11} s'applique.
\end{remark}

Ceci \'etant dit il semble int\'eressant d'\'etudier les d\'eg\'en\'erescences des cas \emph{g\'en\'eriques} -- ce que nous abordons dans les chapitres suivants.

\section{Cas dicritique d\'eg\'en\'er\'e avec $\nu\leq2$}

Dans le paragraphe qui pr\'ec\`ede l'hypoth\`ese ${\rm Cod\ Sing}\ \omega_\nu \geq2$ a \'et\'e essentielle. En fait si $\omega_\nu$ est homog\`ene dicritique avec $\nu\leq2$ et $\omega_\nu$ s'annule sur une hypersurface alors $\nu=2$ et \`a conjugaison lin\'eaire pr\`es $\omega_2$ est de l'un des deux types suivants: $$x_1(x_1dx_2 - x_2dx_1) \quad \mbox{ou} \quad x_3(x_1dx_2 - x_2dx_1).$$ 
Il existe des $1-$formes holomorphes $\omega \in \Omega^1(\mathbb{C}^n,0)$ dont le $2-$jet est $x_1(x_1dx_2 - x_2dx_1)$ pour lesquels il n'existe pas de $f \in \mathcal{O}(\mathbb{C}^2,0)$ tels que $\displaystyle{\frac{\omega}{f}}$ soit ferm\'ee: c'est d'ailleurs g\'en\'erique par la topologie de Krull. Pour ce qui concerne les dimensions sup\'erieures \`a 2 on peut poser la
\vskip0.2cm
\begin{center}
\begin{minipage}{15cm}
{\bf Question:} \emph{Soit $\omega \in \Omega^1(\mathbb{C}^n,0)$ int\'egrable dont le $2$-jet est $x_1(x_1dx_2 - x_2dx_1)$. Existe-t-il $F: \mathbb{C}^n,0 \rightarrow \mathbb{C}^2,0$ et $\omega_0 \in \Omega^1(\mathbb{C}^2,0)$ tels que $\omega\wedge F^*\omega_0 = 0$, i.e. le feuilletage $\mathcal{F}_\omega$ est-il le tir\'e en arri\`ere d'un feuilletage du plan?}
\end{minipage}
\end{center}
\vskip0.2cm
Pour ce qui concerne l'autre mod\`ele nous avons la:

\begin{proposition}\label{prop16}
Il n'existe pas de $1-$forme $\omega \in \Omega^1(\mathbb{C}^n,0)$, $n\geq3$, $\omega$ int\'egrable, ${\rm Cod\ Sing}\ \omega\geq2$ ayant pour $2-$jet $\omega_2 = x_3(x_1dx_2- x_2dx_1)$.
\end{proposition}

\begin{proof} Quitte \`a couper par un $3-$plan g\'en\'eral, il suffit d'\'etablir la proposition pour $n=3$. Supposons qu'il existe $\omega \in \Omega^1(\mathbb{C}^3, 0)$ telle que $j^2\omega = \omega_2$ et ${\rm Cod\ Sing}\ \omega \geq2$. 

Consid\'erons l'\'eclatement $E:\tilde{\mathbb{C}}^3 \rightarrow \mathbb{C}^3,0$ de l'origine de $\mathbb{C}^3$. Le diviseur exceptionnel $E^{-1}(0) \simeq \mathbb{P}^2_{\mathbb{C}}$ est g\'en\'eriquement transverse au feuilletage $\tilde{\mathcal{F}}$ relev\'e de $\mathcal{F} = \mathcal{F}_\omega$ par $E$. En se pla\c cant dans la carte $(x_1,x_2,x_3)$ o\`u $E$ s'exprime sous la forme:
$$E(x_1,x_2,x_3) =(x_1,x_1x_2,x_1x_3) $$ on constate que $\tilde{\mathcal{F}}$ est d\'efini par:
$$\frac{E^*\omega}{x_1^3} = \tilde{\omega}= x_3dx_2 + \tilde{A}dx_1 + x_1\tilde{B} dx_2 + x_1\tilde{C}dx_3. $$

En \'ecrivant explicitement l'int\'egrabilit\'e de $\tilde{\omega}$ on constate que $\tilde{\omega}$ s'annule sur la droite $x_1=x_3=0$ contenue dans $E^{-1}(0)$. Un calcul imm\'ediat montre que:
$$d\tilde{\omega}\Big|_{x_1=x_3=0} = dx_3 \wedge dx_2 + dx_1 \wedge (\cdot) . $$ Ainsi $\tilde{\omega}$ annule un champ $\tilde{Z}$ dont la composant suivant $\displaystyle{\frac{\partial}{\partial x_1}}$ est non nulle le long de $x_1=x_3=0$. Ceci oblige $\tilde{\omega}$ \`a s'annuler sur une surface, contredisant l'hypoth\`ese ${\rm Cod\ Sing}\ \omega \geq2$. \end{proof}

Si l'on consid\`ere l'espace $\mathcal{J}_\nu \subset \Omega^1(\mathbb{C}^n,0)$ des germes de $1-$formes int\'egrables $\omega$ tels que $j^{\nu-1}\omega = 0$, $\omega_\nu = j^\nu \omega$ est homog\`ene de degr\'e $\nu$ dicritique, on h\'erite d'une projection (jet d'ordre $\nu$) $$j^\nu: \mathcal{J}_\nu \rightarrow [\mathcal{J}_\nu]  $$ \`a valeur dans $[\mathcal{J}_\nu] = \{\omega_\nu\ \mbox{homog\`ene int\'egrable dicritique de degr\'e}\ \nu\}$. La Proposition \ref{prop16} indique que cette projection n'est en g\'en\'eral pas surjective.

\begin{remark}
Le groupe des diff\'eomorphismes tangents \`a l'identit\'e Diff$_1(\mathbb{C}^n,0)$ agit sur $\mathcal{J}_\nu$ en pr\'eservant les fibres de $j^\nu$. Le Th\'eor\`eme \ref{theo10} de Camacho-Lins Neto s'interpr\`ete comme suit: si $\underline{\omega}_\nu$ satisfait les hypoth\`eses du Th\'eor\`eme \ref{theo10} alors l'action de Diff$_1(\mathbb{C}^n,0)$ sur la fibre $(j^\nu)^{-1}(\underline{\omega}_\nu)$ est transitive.
\end{remark}

\section{Un lemme utile et le Th\'eor\`eme de Camacho--Lins Neto revisit\'e}

Dans ce chapitre $\omega = \omega_\nu + \cdots \in \Omega^1(\mathbb{C}^3,0)$ est un germe de $1-$forme int\'egrable, ${\rm Cod\ Sing}\ \omega\ge2$. Ils nous sera utile de r\'epr\'esenter les germes consid\'er\'es sur une boule centr\'ee \`a l'origine. Nous la noterons $B$ et son rayon sera suppos\'e suffisament petit pour que le discours qui suit ait un sens. On pose $B^* = B - \{0\}$.

\begin{lemma}\label{lemme18} Supposons qu'il existe un recouvrement ouvert $\{U_k\}$ de $B^*$ et des champs de vecteurs holomorphes $X_k \in \Theta(U_k)$ tels que $\omega\big|_{U_k} = i_{X_k} d\omega\big|_{U_k}$. Il existe alors $X \in \Theta(\mathbb{C}^3,0)$ tel que $\omega = i_X d\omega$.
\end{lemma}

\begin{proof} Soit $Z$ un champ de vecteurs tel que $d\omega = i_Z {\rm vol}$ o\`u ${\rm vol}  = dx_1\wedge dx_2 \wedge dx_3$ est une $3-$forme non nulle de $\mathbb{C}^3$. L'hypoth\`ese du lemme implique que l'ensemble des z\'eros ${\rm Sing}\ Z$ de $Z$ est contenu dans ${\rm Sing}\ \omega$ et donc de codimension au moins deux. Lorsque $U_i \cap U_j \neq \emptyset$, on a $i_{X_i-X_j} d\omega\big|_{U_i\cap U_j} = 0$; puisque ${\rm Cod\ Sing}\ Z\geq2$ il existe $h_{ij}\in \mathcal{O}(U_i\cap U_j)$ tels que $X_i-X_j = h_{ij}Z$. Comme le groupe de cohomologie $H^1(B^*,{\mathcal O})$ est trivial, c'est un r\'esultat de H. Cartan \cite{Car} (voir \cite{Can-Cer-Des}), on a, quitte \`a choisir un bon recouvrement, $h_{ij} = h_i-h_j$, $h_k \in \mathcal{O}(U_k)$. Ainsi les champs $X_i - h_i Z$ se recollent et d\'efinissent un champ global $X^*\in \Theta (B^*)$. Le Th\'eor\`eme de Hartogs permet de prolonger ce champ en un champ $X \in \Theta(B)$ qui v\'erifie $\omega = i_X d\omega$.
\end{proof}

\`A titre d'application nous pr\'esentons une preuve de l'\'enonc\'e de Camacho--Lins Neto, partie $(1)$.

\vskip0.2cm
\noindent {\em D\'emonstration du Th\'eor\`eme \ref{theo10}, partie $(1)$.} Puisque $d\omega_\nu$ ne s'annule qu'en $0$ il en est de m\^eme pour $d\omega$. Par suite en chaque point $m$ d'une boule \'epoint\'ee $B^*$ on peut trouver un champ $X_{,m} \in \Theta(B^*,m)$ tel que $\omega_{,m} = i_{X_{,m}}d\omega_{,m}$. D'apr\`es le Lemme \ref{lemme18} il existe $X \in \Theta(\mathbb{C}^3,0)$ tel que $\omega = i_X d\omega$. Il est clair que $X$ s'annule en $0$ sinon $d\omega_\nu$ s'annulerait sur une droite. Soit $X_1$ la partie lin\'eaire de $X$ en $0$; on a: $$\omega_\nu = i_{X_1}d\omega_\nu = i_{\frac{R}{\nu+1}}d\omega_\nu$$ qui conduit \`a $\displaystyle{i_{X_1 - \frac{R}{\nu+1}} d\omega_\nu = 0}$. Comme $\nu \geq 3$ et $d\omega_\nu$ est \`a singularit\'e isol\'ee on a $\displaystyle{X_1 = \frac{R}{\nu+1}}$. Le Th\'eor\`eme de lin\'earisation des champs de vecteur de Poincar\'e permet alors de supposer que $\displaystyle{X = X_1 = \frac{R}{\nu+1}}$. Un calcul direct (utilisant $L_{\frac{R}{\nu+1}} \omega = \omega$) montre que $\omega = \omega_\nu$. \qed

\begin{remark}
Dans l'article \cite{Cam-Alc} les auteurs substituent au Lemme \ref{lemme18} le Lemme de Division de de Rham-Saito \cite{S}. Il se trouve que le Lemme de Division peut se d\'emontrer (sur $\mathbb{C}$ et en dimension $3$) en utilisant l'argument de H. Cartan.
\end{remark}

\begin{remark}
Dans \cite{Alc} l'un de nous montre que la seule hypoth\`ese ${\rm Sing}\ d\omega = \{0\}$ conduit, \`a conjugaison pr\`es, \`a l'existence d'un champ de vecteurs radial \`a poids $X = \sum p_i x_i \frac{\partial}{\partial x_i}$, $p_i \in \mathbb{N}^*$, tel que $L_X\omega = p\omega$, $p\in \mathbb{N}$, ce qui permet de montrer que $\omega$ est {\em quasi-homog\`eneisable}. Le feuilletage $\mathcal{F}_\omega$ peut \^etre vue comme image r\'eciproque d'un feuilletage sur un espace projectif \`a poids de dimension $2$. 
\end{remark}

Le lemme qui suit va nous permettre de donner une g\'en\'eralisation des r\'esultats qui pr\'ec\`edent. Comme toujours on d\'esigne par $E:\tilde{\mathbb{C}}^3 \rightarrow \mathbb{C}^3,0$ l'\'eclatement de l'origine de $\mathbb{C}^3$; le diviseur exceptionnel $E^{-1}(0)$ est isomorphe \`a $\mathbb{P}^2_\mathbb{C}$.

\begin{lemma}\label{lemme21}
Soit $\omega \in \Omega^1(\mathbb{C}^3,0)$ un germe de $1-$forme int\'egrable \`a l'origine de $\mathbb{C}^3$ dont la partie homog\`ene $\omega_\nu$ est dicritique et satisfait ${\rm Cod\ Sing}\ \omega_\nu \geq 2$. On suppose qu'en chaque point $m \in E^{-1}(0)$ il existe un germe de champ de vecteur $\tilde{X}_{,m}\in \Theta(\tilde{\mathbb{C}}^3,m)$ tel que:
\begin{enumerate}[$(1)$]
\item $\tilde{X}_{,m}$ est tangent au feuilletage $E^{-1}(\mathcal{F}_\omega)$, i.e. $i_{\tilde{X}_{,m}} E^*\omega = 0$.
\item $\tilde{X}_{,m}$ est transverse \`a $E^{-1}(0)$ en tout point $m \in E^{-1}(0)$.
\end{enumerate}
Alors il existe $X \in \Theta(\mathbb{C}^3,0)$ tel que $i_Xd\omega = \omega = L_X\omega$.
\end{lemma}

\begin{proof} Nous allons montrer que les hypoth\`eses du Lemme \ref{lemme18} sont satisfaites. On \'ecrit
$$\omega = \omega_\nu + \omega_{\nu+1} + \cdots = (A_\nu + \cdots) dx_1 + (B_\nu + \cdots)dx_2 + (C_\nu + \cdots)dx_3 $$
o\`u les $A_\nu$, $B_\nu$, $C_\nu$ sont homog\`enes de degr\'e $\nu$ et satisfont:
$$x_1A_\nu + x_2B_\nu + x_3C_\nu = 0. $$ 

Pla\c cons nous dans la carte $(x_1,x_2,x_3)$ de $\tilde{\mathbb{C}}^3$ o\`u l'expression de $E$ est donn\'ee par $ E(x_1,x_2,x_3) = (x_3x_1,x_3x_2,x_3)$. Dans cette carte $E^{-1}(0)$ est donn\'e par $(x_3=0)$ et 
$$E^*\omega = x_3^{\nu+1} \Big[ (A_\nu(x_1,x_2,1) + x_3A') dx_1 + (B_\nu(x_1,x_2,1) + x_3B')dx_2 + C'dx_3 \Big] $$ o\`u les $A'$, $B'$, $C'$ sont holomorphes le long de $x_3=0$.

En chaque point $m \in E^{-1}(0)$, l'hypoth\`ese implique l'existence d'un champ de vecteur $Y_{,m}$:
$$Y_{,m} = \frac{\partial}{\partial x_3} + Y_1\frac{\partial}{\partial x_1} + Y_2 \frac{\partial}{\partial x_2}, \quad Y_i\ \mbox{holomorphes}, $$
tel que $i_{Y_{,m}} E^*\omega_{,m} = 0$. En conjugant localement $Y_{,m}$ \`a $\displaystyle{\frac{\partial}{\partial x_3}}$ par un diff\'eomorphisme local fibr\'e $\phi_{,m} = (\varphi,x_3)$, $\varphi(x_1,x_2,0) = (x_1,x_2)$, on conjugue $E^*\omega_{,m}$ \`a
$$W x_3^{\nu+1} (A_\nu(x_1,x_2,0)dx_1 + B_\nu(x_1,x_2,0) dx_2) $$ o\`u $W$ est une unit\'e, $W(x_1,x_2,0) = 1$. Quitte \`a changer $x_3$ en $W^{\frac{1}{\nu+1}} x_3$, on conjugue $E^*\omega_{,m}$ \`a:
$$\Omega = x_3^{\nu+1} (A_\nu(x_1,x_2,0)dx_1 + B_\nu(x_1,x_2,0) dx_2).$$ Un calcul \'el\'ementaire montre que $\displaystyle{i_{\frac{\partial}{\partial x_3}}} d\Omega = (\nu+1) \Omega$. De sorte que l'on peut modifier le champ $\tilde{X}_{,m}$ pour qu'il satisfasse $\displaystyle{i_{\tilde{X}_{,m}} dE^*\omega_{,m} = E^*\omega_{,m}}$. La collection des $\tilde{X}_{,m}$ nous permet de construire des champs $X_k$ satisfaisant le Lemme \ref{lemme18}.
\end{proof}

Nous allons maintenant d\'ecrire une premi\`ere cons\'equence du Lemme \ref{lemme21}. Une $1-$forme dicritique $\omega_\nu \in \Omega^1(\mathbb{C}^3)$, ${\rm Cod\ Sing}\ \omega_\nu \geq 2$, induit un feuilletage $[\mathcal{F}_{\omega_\nu}]$ de degr\'e $\nu-1$ sur $\mathbb{P}^2_\mathbb{C}$. Ce feuilletage poss\`ede exactement $\nu^2 - \nu + 1$ points singuliers compt\'es avec multiplicit\'e. Nous dirons que $[{\mathcal F}_{\omega_\nu}]$ est {\em sans nilpotence} si en tout point $m \in {\rm Sing}\ [{\mathcal F}_{\omega_\nu}]$ le feuilletage est d\'efini par une $1-$forme locale $\alpha_{,m}$ dont la partie lin\'eaire est non nilpotente (en particulier non nulle). Dit autrement $\alpha_{,m}$ est ou bien de type Kupka-Reeb, i.e. $d\alpha(m) \neq 0$, ou bien le $1-$jet de $\alpha$ en $m$ est \`a conjugaison pr\`es $d(x_1x_2)$. Dans ce dernier cas nous dirons que $m \in {\rm Sing}\ [{\mathcal F}_{\omega_\nu}]$ est de {\em type centre} ou {\em central} si $[{\mathcal F}_{\omega_\nu}]_{,m}$ poss\`ede une int\'egrale premi\`ere non constante: $\alpha_{,m} = gdf$, $f\in {\mathcal O}({\mathbb P}_{\mathbb C}^2,m)$, $g \in {\mathcal O}^*({\mathbb P}_{\mathbb C}^2,m)$. Notons qu'un tel $f$ est n\'ecessairement de Morse en $m$.

\begin{proposition}\label{prop22}
Soit $\omega = \omega_\nu + \cdots \in \Omega^1({\mathbb C}^3,0)$ un germe de $1-$forme int\'egrable \`a l'origine de ${\mathbb C}^3$ dont la partie homog\`ene $\omega_\nu$ est dicritique, $\nu \geq3$ et ${\rm Cod\ Sing}\ \omega_\nu \geq2$. Si $[{\mathcal F}_{\omega_\nu}]$ est sans nilpotence et n'a pas de point central, alors $\omega$ et $\omega_\nu$ sont conjugu\'es; en particulier ${\mathcal F}_{\omega}$ et ${\mathcal F}_{\omega_\nu}$ sont conjugu\'es.
\end{proposition}

\begin{proof} D'apr\`es la Proposition \ref{prop3} les hypoth\`eses du Lemme \ref{lemme21} sont satisfaites. Il existe donc un champ $X \in \Theta(\mathbb{C}^3,0)$ tel que $i_X(d\omega) = \omega$. On conclut comme dans la preuve du Th\'eor\`eme \ref{theo10}: comme $\displaystyle{\omega_\nu = i_{\frac{R}{\nu+1}} d\omega_\nu}$ et ${\rm Cod\ Sing}\ \omega_\nu \geq 2$ on a ${\rm Cod\ Sing}\ d\omega_\nu \geq 2$. Par suite si $X_1 = j^1X$, l'\'egalit\'e $\displaystyle{i_{X_1 - \frac{R}{\nu+1}} d\omega_\nu} = 0$ implique que $\displaystyle{X_1 = \frac{R}{\nu+1}}$. On conclut en lin\'earisant $X$.
\end{proof}

\begin{remark}
La condition sans nilpotence, sans point central est g\'en\'erique dans l'espace des $1-$formes dicritiques de degr\'e $\nu$.
\end{remark}

On peut interpr\'eter la Proposition \ref{prop22} comme un \'enonc\'e de rigidit\'e ou de d\'e\-ter\-mi\-na\-tion finie.

\begin{proposition}
Soit $\omega = \omega_\nu + \cdots \in \Omega^1({\mathbb C}^3,0)$ un germe de $1-$forme int\'egrable \`a l'origine de ${\mathbb C}^3$ dont la partie homog\`ene $\omega_\nu$ est dicritique, $\nu \geq 3$ et ${\rm Cod\ Sing}\ \omega_\nu \geq2$. On suppose qu'en chaque point singulier $m$ de $[{\mathcal F}_{\omega_\nu}]$ celui ci est donn\'e par une $1-forme$ \`a $1-$jet non nul. Si ${\rm Sing}\ \omega$ est hom\'eomorphe \`a ${\rm Sing}\ \omega_\nu$, alors $\omega$ et $\omega_\nu$ sont conjugu\'es. 
\end{proposition}

\begin{proof} Laiss\'ee au lecteur. \end{proof}

\section{Feuilletages avec singularit\'es nilpotentes; probl\`emes de d\'eploiements}

Dans ce chapitre on identifie ${\mathbb C}^3$ \`a ${\mathbb C}^2 \times {\mathbb C}$ et l'on consid\`ere un feuilletage ${\mathcal F}_0$ sur ${\mathbb C}^2,0$ ayant une singularit\'e nilpotente non triviale en $0$, i.e. avec $1-$jet nilpotent non nul. Soit ${\mathcal F}$ un d\'eploiement de ${\mathcal F}_0$ sur ${\mathbb C}^3$, c'est \`a dire un germe de feuilletage en $0 \in {\mathbb C}^3$ dont la restriction \`a ${\mathbb C}^2 \times \{0\}$ est ${\mathcal F}_0$. Si $\omega \in \Omega^1({\mathbb C}^3,0)$ est une $1-$forme d\'efinissant ${\mathcal F}$, le feuilletage ${\mathcal F}_0$ est d\'efini par $\omega_0 \in \Omega^1({\mathbb C}^2,0)$ restriction de $\omega$ \`a ${\mathbb C}^2 \times \{0\}$. On fixe des coordonn\'ees $(x_1,x_2,x_3)$ de ${\mathbb C}^3 = {\mathbb C}^2 \times {\mathbb C}$; comme ${\mathcal F}_0$ est \`a singularit\'e nilpotente non triviale on peut supposer que le $1-$jet de $\omega_0$ est $x_1dx_1$. Il se pourrait que $\omega$ soit non singuli\`ere en $0$, c'est le cas par exemple pour le feuilletage par les niveaux de $x_1^2 + x_2^3 + x_3$; mais dans ce cas ${\mathcal F}$ et donc ${\mathcal F}_0$ poss\`ede une int\'egrale premi\`ere holomorphe non constante (Frobenius classique).

Supposons maintenant $\omega$ singuli\`ere en $0$. Le $1-$ jet de $\omega$ est du type:
$$\omega_1 = x_1dx_1 + x_3 dl + Ldx_3 $$ o\`u $l = l(x_1,x_2)$ et $L = L(x_1,x_2,x_3)$ sont des formes lin\'eaires. L'int\'egrabilit\'e de $\omega$ impliquant celle de $\omega_1$ on a: 
$$ (x_1dx_1 + x_3dl) \wedge d(l-L) \wedge dx_3 \equiv 0 $$ qui implique $d\omega(0) = d\omega_1(0) = \lambda dx_1\wedge dx_3$, $\lambda \in {\mathbb C}$.

Il y a deux cas suivant que $\lambda$ est nul ou non. Si $d\omega_1(0)\neq 0$ alors $l = l(x_1)$, $L=L(x_1,x_3)$ et nous avons un ph\'enom\`ene de Kupka-Reeb: il existe une submersion $$F: (x_1,x_2,x_3) \rightarrow (X_1(x_1,x_2,x_3), X_3(x_1,x_2,x_3)),\quad F(x_1,0,x_3) = (x_1,x_3),$$ telle que $\omega = F^*\alpha$ o\`u $\alpha$ est la restriction de $\omega$ au plan $x_2=0$. Si $$\omega = A dx_1 + Bdx_2 + Cdx_3 = x_1dx_1 + \cdots$$ on a $F^*\alpha = A(X_1,0,X_3)dX_1 + C(X_1,0,X_3)dX_3$. Ainsi $\omega_0$ est la restriction de $F^*\alpha$ au plan $x_3=0$. Un calcul direct \'el\'ementaire montre que le {\em nombre de Milnor} $\mu({\mathcal F}_0;0)$ de ${\mathcal F}_0$ en $0$ est plus grand ou \'egal \`a $3$. Dans ce cas il y a un facteur int\'egrant pour ${\mathcal F}$, et donc ${\mathcal F}_0$ (tout du moins formel). 

Supposons maintenant $\omega_1$ ferm\'ee: $$\omega_1 = d\left( \frac{x_1^2}{2} + x_3l  + \varepsilon x_3^2 \right) = dq, \quad \varepsilon \in {\mathbb C},\quad l = l(x_1,x_2)\ \mbox{lin\'eaire}. $$ De nouveau nous distinguons trois cas suivant le rang de la forme quadratique $q$.

Si $q$ est de rang $3$, la forme $\omega$ est \`a singularit\'e isol\'ee et les feuilletages ${\mathcal F}$ et donc ${\mathcal F}_0$ poss\`edent une int\'egrale premi\`ere holomorphe non constante (Th\'eor\`eme de Malgrange).

Si $q$ est de rang $2$, ${\displaystyle q = \frac{x_1^2}{2} + x_3(ax_1 + \varepsilon x_3)}$, alors d'apr\`es la Proposition \ref{prop3} ou bien ${\mathcal F}$ et donc ${\mathcal F}_0$ poss\`edent une int\'egrale premi\`ere holomorphe ou bien il existe une submersion $F: {\mathbb C}^3,0 \rightarrow {\mathbb C}^2,0 = (x_2=0)$, $$F:(x_1,x_2,x_3) \rightarrow (X_1(x_1,x_2,x_3), X_3(x_1,x_2,x_3)), $$ $F(x_1,0,x_3) = (x_1,x_3)$, dont les fibres sont tangents \`a ${\mathcal F}$. Avec les m\^emes notations que pr\'ec\'edemment le feuilletage ${\mathcal F}$ est donn\'e par $A(X_1,0,X_3)dX_1 + C(X_1,0,X_3)dX_3$ et ${\mathcal F}_0$ par $A(X_1,0,X_3)dX_1 + C(X_1,0,X_3)dX_3\Big|_{x_3=0}$. Comme ${\mathcal F}_0$ poss\`ede facteur int\'egrant formel, il en est de m\^eme pour ${\mathcal F}$ \cite{Cer-Ma,Cer-Mo}. Un calcul \'el\'ementaire montre que le nombre de Milnor de cette derni\`ere forme est plus grand que trois.

Avant d'attaquer le cas restant o\`u $q$ est de rang $1$, nous r\'esumons la situation dans la

\begin{proposition}\label{prop24}
Soit ${\mathcal F} = {\mathcal F}_\omega$ un d\'eploiement de ${\mathcal F}_0$ sur ${\mathbb C}^3,0 = {\mathbb C}^2,0 \times {\mathbb C}, 0$, ${\mathcal F}_0 = {\mathcal F}\big|_{{\mathbb C}^2,0 \times \{0\}}$. On suppose que ${\mathcal F}_0$ a une singularit\'e nilpotente non triviale en $0$. On est dans l'un des cas suivants:
\begin{enumerate}[$(1)$]
\item ${\mathcal F}$ et ${\mathcal F}_0$ ont une int\'egrale premi\`ere holomorphe non constante.
\item Le nombre de Milnor $\mu({\mathcal F}_0;0)$ est plus grand que $3$ et ${\mathcal F}$ et ${\mathcal F}_0$ ont un facteur int\'egrant (formel).
\item Le $1-$jet de $\omega$ en $0$ est $(x_1+\delta x_3) d(x_1+\delta x_3)$, $\delta \in {\mathbb C}$. \`A conjugaison lin\'eaire pr\`es on peut supposer que $\delta = 0$.
\end{enumerate}
\end{proposition}

Le cas restant correspondant au cas $(3)$ o\`u rang $q=1$ se traite de la fa\c con suivante, tout de moins lorsque $\mu({\mathcal F}_0;0) = 2$.

\begin{proposition}\label{prop25}
Soit ${\mathcal F} = {\mathcal F}_\omega$ un d\'eploiement de ${\mathcal F}_0 = {\mathcal F}_{\omega_0} ={\mathcal F}\big|_{{\mathbb C}^2,0 \times \{0\}}$; on suppose que le $1-$ jet de $\omega$  est $j^1\omega = x_1dx_1$. Si $\mu({\mathcal F}_0;0) = 2$, alors on est dans l'une des situations suivantes: 
\begin{enumerate}[$(1)$]
\item ${\mathcal F}_0$ (et ${\mathcal F}$) ont une int\'egrale premi\`ere holomorphe.
\item Le d\'eploiement ${\mathcal F}$ de ${\mathcal F}_0$ est trivial, plus pr\'ecis\'ement il existe un champ de vecteurs $\xi \in \Theta({\mathbb C}^3,0)$ tangent \`a ${\mathcal F}$ et transverse \`a ${\mathbb C}^2,0 \times \{0\}$.
\end{enumerate}
\end{proposition}

\begin{proof} Elle repose sur la forme normale de Loray. En fait l'\'enonc\'e \ref{loray} dit que l'on peut supposer \`a unit\'e multiplicative pr\`es que $$\omega = X_1dX_1 + (l_1(f) + X_1 l_2(f))df $$ avec $f = f(x_2,x_3) \in {\mathcal O}({\mathbb C}^2,0)$, $l_1$ et $l_2 \in {\mathcal O}({\mathbb C},0)$ et $X_1 = x_1 + \cdots \in {\mathcal O}({\mathbb C}^3,0)$.

Le feuilletage ${\mathcal F}_0$ est alors donn\'e par $$\omega_0 = X_1(x_1,x_2,0)dX_1(x_1,x_2,0) + (l_1(f(x_2,0))+ X_1 l_2(f(x_2,0)) )df(x_2,0) $$ qui est conjugu\'e \`a $${\omega}'_0 = x_1dx_1 + (l_1(f_0(x_2)) + x_1l_2(f_0(x_2))) df_0(x_2)$$ avec $f_0(x_2) = f(x_2,0)$. Si $f_0$ est une submersion alors $(X,f)$ aussi et l'on est dans le cas $(2)$. On peut donc supposer que $f_0(x_2) = x_2^k$ avec $k\geq 2$. 

Remarquons que si $l_1$ est une unit\'e alors ${\omega}'_0$ poss\`ede une int\'egrale premi\`ere holomorphe et nous sommes dans le cas $(1)$. En effet ${\omega}'_0$ apparait comme image r\'eciproque par l'application $(x_1,x_2^k)$ de 
$${\omega}''_0 = x_1dx_1 + (l_1(u) + x_1l_2(u))du $$ qui se trouve non singuli\`ere, donc avec int\'egrale premi\`ere. On peut donc supposer que $l_1 = \varepsilon u^l + \cdots$, $l\geq1$, $\varepsilon \neq 0$; l'id\'eal des composantes de ${\omega}'_0$ est donc $<x_1,x_2^{lk+k-1}>$. Comme $k\geq 2$ et $l\geq 1$, $k(l+1)-1\geq 3$ et donc $\mu({\mathcal F}_0;0)\geq 3$, cas exclus par hypoth\`ese.
\end{proof}

Dans \cite{Cer-Des} on classifie les feuilletages de degr\'e $2$ du plan ${\mathbb P}^2_{\mathbb C}$ n'ayant qu'une seule singularit\'e. Il d\'ecoule de cette classification que si $m$ est un point singulier de type nilpotent non trivial d'un feuilletage ${\mathcal F}_0$ de degr\'e $2$ alors $\mu({\mathcal F}_0;m)\leq 6$.

En fait la Proposition \ref{prop25} s'\'etend comme suit:

\begin{theorem}\label{theo26}
Soit ${\mathcal F} = {\mathcal F}_\omega$ un feuilletage de ${\mathbb C}^3,0$ d\'eploiement de ${\mathcal F}_{\omega_0} = {\mathcal F}_0 = {\mathcal F}\big|_{{\mathbb C}^2,0 \times \{0\}}$; on suppose que le $1-$jet de $\omega$ est $x_1dx_1$.
\begin{enumerate}[$(1)$]
\item Si $\mu({\mathcal F}_0;0) = 3$ alors ou bien le d\'eploiement ${\mathcal F}$ est trivial ou bien ${\mathcal F}_0$ (et donc ${\mathcal F}$) poss\`ede un facteur int\'egrant formel $g_0 \in \widehat{{\mathcal O}({\mathbb C}^2,0)}$ (respectivement $g \in {\mathcal O}({\mathbb C}^3,0)$) tel que ${\displaystyle d\Big( \frac{\omega_0}{g_0}\Big) = 0}$ (respectivement ${\displaystyle d\Big( \frac{\omega}{g}\Big) = 0}$).
\item Si $\mu({\mathcal F}_0;0)+1$ est un nombre premier alors ou bien le d\'eploiement ${\mathcal F}$ est trivial ou bien ${\mathcal F}_0$ (et donc ${\mathcal F}$) poss\`ede une int\'egrale premi\`ere holomorphe non constante.
\end{enumerate} 
\end{theorem}

\begin{proof} Elle proc\`ede comme la pr\'ec\'edente; lorsque le d\'eploiement est non trivial on a, avec les m\^emes notations:
$$\omega'_0 = x_1dx_1 + (l_1(x_2^k) + xl_2(x_2^k))kx_2^{k-1}dx_2, \quad k\geq 2. $$
Visiblement $\omega'_0$ est image r\'eciproque par $(x_1,x_2^k)$ de:
$$\omega''_0 = x_1dx_1 + (l_1(u) + x_1l_2(u))du. $$
Comme dans la preuve pr\'ec\'edente si $l_1$ est unit\'e $\omega''_0$ poss\`ede une int\'egrale premi\`ere holomorphe et par suite ${\mathcal F}_0$ et ${\mathcal F}$ aussi. Sinon le $1-$jet de $\omega''_0$ est $x_1dx_1 + (\varepsilon u + \lambda x_1)du $ avec $\varepsilon = l'_1(0)$ et $\lambda = l_2(0)$.

Ce $1-$jet est nilpotent si et seulement si $\varepsilon = \lambda = 0$; s'il est non nilpotent, la th\'eorie \'el\'ementaire des formes normales \cite{Cer-Ma} assure que $\omega''_0$ (et donc $\omega'_0$ et puis $\omega$) poss\`ede un facteur int\'egrant formel. On pose $l_1(u) = \varepsilon_p u^p + \cdots$, $\varepsilon_p \neq 0$, de sorte que $$l_1(x_2^k) = \varepsilon_p x_2^{kp} + \cdots, \quad k \geq 2.$$ L'id\'eal des composantes de $\omega'_0$ est donc $<x_1,x_2^{k(p+1)-1}>$ et par suite $\mu({\mathcal F}_0;0)+1 = k(p+1)$. Comme $k\geq2$, si $\mu({\mathcal F}_0;0)+1$ est premier alors $p=0$ et $l_1$ est une unit\'e; de sorte que ${\mathcal F}_0$ a une int\'egrale premi\`ere holomorphe non constante. Lorsque $\mu({\mathcal F}_0;0) = 3$, alors ou bien $(k,p) = (4,0)$ et ${\mathcal F}_0$ a une int\'egrale premi\`ere holomorphe, ou bien $(k,p) = (2,1)$, cas o\`u ${\mathcal F}_0$ a un facteur int\'egrant formel. \end{proof}

\begin{remark}\label{rem11}
Comme on l'a dit, pour une singularit\'e nilpotente d'un feuilletage de degr\'e $2$ on a que $\mu \in \{2,3,4,5,6\}$; seul le nombre $5$ n'apparait pas dans la liste du Th\'eor\`eme \ref{theo26}. Dans ce cas avec les notations pr\'ec\'edentes on a, lorsque $k\geq2$, les possibilit\'es $(k,p) = (6,0)$ (pr\'esence d'int\'egrale premi\`ere holomorphe), $(k,p) = (3,1)$ (pr\'esence de facteur int\'egrant formel) mais aussi \`a priori $(k,p)=(2,2)$; dans ce cas, si $\lambda \neq 0$, il y a encore un facteur int\'egrant formel.
\end{remark}

\begin{remark}
Le cas $\mu = 2$, trait\'e dans la Proposition \ref{prop25}, l'est aussi dans le Th\'eor\`eme \ref{theo26}. Nous avons souhait\'e s\'eparer ces deux cas pour des raisons de lisibilit\'e. Remarquons aussi que l'\'enonc\'e \ref{theo26} se g\'en\'eralise en dimension $n$ quelconque avec une preuve identique.
\end{remark}

\noindent {\em Exemple:} Soit ${\mathcal F}_0$ le feuilletage de degr\'e $2$ d\'efini par $\omega_0 = (x_1-x_2^3)dx_1 + x_1x_2^2 dx_2 = (x_1 - x_2^3)dx_1 + \frac{1}{3} x_1dx_2^3$. On a $\mu({\mathcal F}_0;0)=5$, $k=3$, $p=1$ et ${\mathcal F}_0$ poss\`ede l'int\'egrale premi\`ere ${\displaystyle \frac{3x_1 - 2x_2^3}{x_1^3}}$.

\section{Th\'eor\`eme de Dulac: pr\'esentation et applications}\label{section6}

Rappelons un r\'esultat c\'el\`ebre d'Henri Dulac \cite{D} sous sa forme g\'en\'eralis\'ee dans \cite{Cer-Alc}:

\begin{theorem}\label{theo29}
Soit ${\mathcal F}$ un feuilletage de degr\'e deux sur le plan projectif ${\mathbb P}^2_{\mathbb C}$. Supposons qu'il existe un point singulier $m \in {\rm Sing}\ {\mathcal F}$ de type centre. Alors ${\mathcal F}$ est d\'efini par une $1-$forme ferm\'ee rationelle $\eta$.
\end{theorem}

En fait Dulac donne en carte affine une liste de formes normales pour les feuilletages de degr\'e $2$ ayant une singularit\'e de type centre sous l'hypoth\`ese que la {\em droite \`a l'infini} est invariante (il consid\`ere en effet des $1-$formes de type $Adx + Bdy$ o\`u $A$ et $B$ sont des polyn\^omes de degr\'e inf\'erieur ou \'egal \`a $2$). Dans \cite{Cer-Alc} on montre que la pr\'esence d'une singularit\'e de type centre implique celle d'une droite invariante, ce qui permet d'invoquer Dulac. Les preuves reposent sur des calculs formels longs typiques du d\'ebut du 20e si\`ecle et l'intuition d'un tel r\'esultat reste myst\'erieuse; il serait int\'eressant d'ailleurs d'en produire une d\'emonstration g\'eom\'etrique.

Ce th\'eor\`eme doit \^etre interpr\'et\'e comme suit: une certaine condition d'int\'egrabilit\'e {\em locale} (condition de centre) conduit \`a un r\'esultat d'int\'egrabilit\'e {\em globale} ({\em primitive} d'une forme ferm\'ee rationelle).

Voici la liste des mod\`eles obtenus pas H. Dulac. Les feuilletages {\em g\'en\'eriques} ont une singularit\'e de type centre. Par contre il y a des feuilletages qui n'ont pas de singularit\'es centrales; ces deux situations sont contenues dans la liste ci-dessous.

\begin{theorem}{\rm(Dulac \cite{D})}
La $1-$forme ferm\'ee rationelle $\eta$ est d'un des types suivants:
\begin{enumerate}[{\rm(a)}]
\item $\eta = dq$, o\`u $q$ est un polyn\^ome de degr\'e $3$.
\item $\eta = \sum_{i=1}^{3} \lambda_j \frac{dp_j}{p_j}$, o\`u $\lambda_j \in {\mathbb C}^*$ et $p_j$ est un polyn\^ome de degr\'e $1$, $1\leq j \leq 3$.
\item $\eta = \sum_{i=1}^{2} \lambda_j \frac{dp_j}{p_j}$, o\`u $\lambda_j \in {\mathbb C}^*$, $j=1,2$, {\rm deg}$(p_1) = 2$ et {\rm deg}$(p_2)=1$.
\item $\eta = \sum_{i=1}^{2} \lambda_j \frac{dp_j}{p_j} + dq$, o\`u $\lambda_j \in {\mathbb C}^*$, {\rm deg}$(p_j) = 1$, $j=1,2$, et {\rm deg}$(q)=1$.
\item $\eta = \sum_{i=1}^{2} \lambda_j \frac{dp_j}{p_j} + d\big(\frac{q}{p_1}\big)$, o\`u $\lambda_1$, $\lambda_2$, $p_1$, $p_2$, $q$ sont comme en {\rm (d)}.
\item $\eta = \frac{dp}{p} + d \big( \frac{q}{p^2} \big)$, o\`u {\rm deg}$(p)=1$ et {\rm deg}$(q)=2$.
\item $\eta = \frac{dp}{p} + d \big( \frac{q}{p} \big)$, o\`u $p$ et $q$ sont comme en {\rm (f)}.
\item $\eta = \frac{dp}{p} + dq$, o\`u $p$ et $q$ sont comme en {\rm (f)}.
\item $\eta = \frac{dp}{p} + dq$, o\`u {\rm deg}$(p)=2$ et {\rm deg}$(q)=1$.
\item $\eta = 3\frac{df}{f} - 2\frac{dg}{g}$, o\`u {\rm deg}$(f)=2$, {\rm deg}$(g)=3$ et $3gdf - 2fdg$ est divisible par une fonction affine.
\end{enumerate}
\end{theorem}

Voici une cons\'equence du Th\'eor\`eme de Dulac:

\begin{theorem}\label{theo30}
Soit $\omega = \omega_3 + \cdots \in \Omega^1({\mathbb C}^3,0)$ un germe de $1-$forme int\'egrable \`a l'origine de ${\mathbb C}^3$ dont la partie homog\`ene $\omega_3$ est dicritique et ${\rm Cod\ Sing}\ \omega_3 \geq2$. Soit $[{\mathcal F_{\omega_3}}]$ le feuilletage de degr\'e deux de ${\mathbb P}^2_{\mathbb C}$ induit par ${\omega_3}$. 
\begin{enumerate}[$(1)$]
\item Si $[{\mathcal F}_{\omega_3}]$ poss\`ede un point central alors ${\mathcal F}_\omega$ est d\'efini par une $1-$forme ferm\'ee m\'eromorphe $\eta$.
\item Si $[{\mathcal F}_{\omega_3}]$ ne poss\`ede pas de point singulier de type centre ou nilpotent, alors ${\mathcal F}_\omega$ est conjugu\'e au feuilletage homog\`ene ${\mathcal F}_{\omega_3}$.
\end{enumerate}
\end{theorem}

\begin{proof}
Si $[{\mathcal F}_{\omega_3}]$ satisfait le point $(2)$ on applique la Proposition \ref{prop22}. Si $[{\mathcal F}_{\omega_3}]$ a un point central on invoque le Th\'eor\`eme de Dulac: $[{\mathcal F}_{\omega_3}]$ est d\'efini par une $1-$forme ferm\'ee rationelle $\eta$. Soit $E: {\tilde {\mathbb C}^3},0 \rightarrow {\mathbb C}^3,0$ l'\'eclatement de l'origine. La restriction $E^*{\underline \eta}\big|_{E^{-1}(0)}$ d\'efini le feuilletage $[{\mathcal F}_{\omega_3}]$ sur $E^{-1}(0) \simeq {\mathbb P}^2_{\mathbb C}$. Le Th\'eor\`eme d'extension de L\'evi permet de construire une $1-$forme m\'eromorphe ${\tilde \eta}$ au voisinage de $E^{-1}(0)$ et d\'efinissant $E^{-1}({\mathcal F})$. Toujours par L\'evi, il existe $\eta$ m\'eromorphe \`a l'origine de ${\mathbb C}^3$ telle que $E^*\eta = {\tilde \eta}$. Cet $\eta$ convient.
\end{proof}

Avant de s'int\'eresser \`a des centres {\em d\'eg\'en\'er\'es} mentionnons que l'\'enonc\'e \ref{theo29} ne se g\'en\'eralise pas en degr\'e plus grand que deux. En voici la raison; consid\'erons un feuilletage ${\mathcal F}'$ de degr\'e $\geq2$ en ${\mathbb P}^2_{\mathbb C}$. Si ${\mathcal F}'$ est g\'en\'erique il ne poss\`ede pas de courbe alg\'ebrique invariante (\cite{J, Cer-Alc}) et par suite il ne peut \^etre d\'efini par une $1-$forme ferm\'ee. Choisissons une carte affine ${\mathbb C}^2 = \{(x_1,x_2)\} \subset {\mathbb P}^2_{\mathbb C}$ telle que l'origine $(0,0)$ soit un point non singulier de ${\mathcal F}'$ et telle que la feuille passant par $(0,0)$ soit tangente \`a l'axe $x_2=0$ en $(0,0)$; en d'autre termes ${\mathcal F}'$ est d\'efini \`a l'origine par une $1-$forme du type $dx_2 + \cdots$. Consid\'erons l'application birationelle d\'efinie par $\sigma(x_1,x_2) = (x_2,x_1x_2)$; le feuilletage $\sigma^{-1}({\mathcal F}') = {\mathcal F}$ a une int\'egrale premi\`ere de Morse en $(0,0)$, c'est donc un centre, mais n'est pas d\'efini par une $1-$forme ferm\'ee rationnelle puisque ${\mathcal F}'$ ne l'est pas.

\section{Points nilpotents  g\'en\'eriques}

Nous nous proposons de pr\'eciser le r\'esultat du Th\'eor\`eme \ref{theo30} en \'etablissant d'abord un th\'eor\`eme de type Dulac dans le cas nilpotent g\'en\'erique. Une d\'emonstration de ce m\^eme r\'esultat a \'et\'e propos\'ee par C. Rousseau et C. Christopher suite \`a une discussion avec l'un des auteurs (communication personnelle).

\begin{theorem}\label{theo31}
Soit ${\mathcal F}$ un feuilletage de degr\'e deux sur le plan projectif ${\mathbb P}^2_{\mathbb C}$. On suppose que ${\mathcal F}$ poss\`ede un point singulier $m$ nilpotent tel que $\mu({\mathcal F},m) = 2$. Si ${\mathcal F}_{,m}$ poss\`ede une int\'egrale premi\`ere holomorphe non constante, alors ${\mathcal F}$ est d\'efini par une $1-$forme ferm\'ee rationelle.
\end{theorem}

\begin{proof} Nous choisissons une carte affine ${\mathbb C}^2 = \{(x_1,x_2)\} \subset {\mathbb P}^2_{\mathbb C}$ dont l'origine est le point $m$. Le feuilletage ${\mathcal F}$ est donc d\'efini par une $1-$forme du type:
$$\omega = 2x_1dx_1 + A_2dx_1 + B_2dx_2 + q(x_1dx_2 - x_2dx_1) $$ o\`u les $A_2$, $B_2$ et $q$ sont des polyn\^omes de degr\'e $2$.

Comme le germe ${\mathcal F}_{,0}$ a une int\'egrale premi\`ere holomorphe non triviale, c'est une {\em courbe g\'en\'eralis\'ee} au sens de \cite{Cam-N-S}. En particulier si $f \in {\mathcal O}({\mathbb C}^2,0)$ est une int\'egrale premi\`ere minimale de ${\mathcal F}_{,0}$ (\cite{Ma-Mo}), alors la multiplicit\'e alg\'ebrique de $f$ est $2$ et son nombre de Milnor $\mu(f;0) = 2$ aussi. En r\'esulte que $f$ est holomorphiquement conjugu\'e \`a $x_1^2 - x_2^3$. Ceci nous permet de supposer, \`a automorphisme de ${\mathbb P}^2_{\mathbb C}$ fixant $0$ pr\`es, que
$$ f = x_1^2 -  x_2^3 - \sum\limits_{i\geq4} f_i = \sum\limits_{i\geq2}f_i$$ o\`u les $f_i$ sont des polyn\^omes homog\`enes de degr\'e $i$; on note 
$$f_k = \sum\limits_{i+j=k} a_{ij}x_1^ix_2^j. $$ Remarquons que l'on peut, quitte \`a changer $f$ en $l(f)$, $l \in {\rm Diff}({\mathbb C},0)$, supposer que $a_{2i,0} = 0$ pour tout $i\geq 2$.

Le but est maintenant d'obtenir des renseignements sur les polyn\^omes $A_2$, $B_2$ en exploitant formellement l'\'egalit\'e
\begin{equation}\label{eq2}
0 = \omega \wedge df = (2x_1 + A_2 - qx_2) \left(-3x_2^2 + \frac{\partial f_4}{\partial x_2} + \cdots \right) - (B_2 + qx_1) \left( 2x_1 + \frac{\partial f_4}{\partial x_1} + \cdots \right).
\end{equation}
Visiblement on a $B_2 = -3x_2^2$; en consid\'erant les termes de degr\'e $4$ dans (\ref{eq2}) on constate que:
$$ 2x_1 \frac{\partial f_4}{\partial x_2} - 3x_2^2 A_2 - 2x_1^2q = 0$$ et $A_2$ est divisible par $x_1$:
$$A_2 = x_1(ax_1 + bx_2),\quad a,b\in {\mathbb C}. $$

Si l'on pose $q = Px_1^2 + Qx_1x_2 + Rx_2^2$ on obtient (en utilisant la remarque $a_{2i,0} = 0$):
$$f_4 = \frac{3}{8}bx_2^4 + \left(\frac{a}{2} + \frac{R}{3} \right)x_1x_2^3 + \frac{Q}{2}x_1^2x_2^2 + Px_1^3x_2. $$
L'\'etude des termes d'ordre $5$ dans (\ref{eq2}) donne ${\displaystyle R = - \frac{3a}{8}}$ et
$$f_5 = \left(-\frac{3b^2}{20} - \frac{3Q}{5} \right)x_2^5 - \left(\frac{21}{64} ab + \frac{3P}{2} \right)x_2^4x_1 - \left(\frac{3a^2}{16} + \frac{Qb}{6} \right)x_2^3x_1^2 - \frac{(Qa+Pb)}{4}x_2^2x_1^3 - \frac{Pa}{2}x_1^4x_2+a_{50}x_1^5. $$
On fait ensuite $x_2=0$ dans les termes d'ordre $6$ de (\ref{eq2}) et l'on obtient ${\displaystyle P = -\frac{3}{32}ab}$. Finalement on a
$$A_2 = x_1(ax_1+bx_2); \quad B_2 = -3x_2^2;\quad q = -\frac{3}{32}abx_1^2 +Qx_1x_2 - \frac{3a}{8}x_2^2 $$
et nous allons distinguer deux cas suivant que $a$ est nul ou non.

\vskip0.2cm
\noindent {\em Cas $a=0$.} La $1-$forme $\omega$ s'\'ecrit donc:
$$\omega = x_1(2+bx_2 - Qx_2^2)dx_1 + (Qx_1^2x_2 - 3x_2^2)dx_2 = \Big(1+ \frac{b}{2}x_2 - \frac{Q}{2}x_2^2\Big)dx_1^2 + (Qx_1^2x_2 - 3x_2^2)dx_2 $$ et apparait comme image r\'eciproque de
$$\omega'=   \Big(1+ \frac{b}{2}y - \frac{Q}{2}y^2\Big)dx + (Qxy - 3y^2)dy$$ par l'application $\sigma:(x_1,x_2)\rightarrow(x_1^2,x_2)$. Remarquons que toute forme de type $\omega'$ pr\'ec\'edente, puisque no singuli\`ere poss\`ede une int\'egrale premi\`ere holomorphe et donc $\omega$ aussi. Ainsi l'existence d'une int\'egrale premi\`ere ne donne pas plus d'information sur les coefficients $b$ et $Q$.

Un calcul direct montre que la conique d'\'equation ${\mathcal C} = bQx - 3Qy^2 + 6 = 0$ est invariante par $\omega'$. Par suite ${\mathcal C} \circ \sigma = 0$ est invariante par $\omega$. On v\'erifie que $ \omega' / \Big(1+ \frac{b}{2}y - \frac{Q}{2}y^2\Big) {\mathcal C}$ est ferm\'ee et par suite que $ \omega / \Big(1+ \frac{b}{2}x_2 - \frac{Q}{2}x_2^2\Big) (bQx_1^2 - 3Qx_2^2 + 6)$ aussi; ceci d\'emontre le th\'eor\`eme lorsque $a=0$.

\vskip0.2cm
\noindent {\em Cas $a\neq0$.} On peut supposer $a=1$, \`a isomorphisme pr\`es. Donc
$$\omega = \left[2x_1 + x_1(x_1+bx_2) - x_2 \Big( -\frac{3}{32}bx_1^2+Qx_1x_2 - \frac{3}{8}x_2^2 \Big)\right]dx_1$$ $$ + \left[-3x_2^2 + x_1 \Big( -\frac{3}{32}bx_1^2+Qx_1x_2 - \frac{3}{8}x_2^2 \Big) \right] dx_2.$$ Il faut ici calculer explicitement les coefficients $a_{05}$ et $a_{15}$, ce qui se fait sans difficult\'e \`a partie de (\ref{eq2}): $$a_{15} = \frac{147}{1280}b^2 + \frac{3Q}{5},\quad a_{15} + \frac{5}{8}a_{05} = 0 $$ pour en d\'eduire que ${\displaystyle Q = \frac{3}{32} b^2}$. 

On constate alors que la droite $l = x_1 + bx_2 + 8=0$ est invariante par $\omega$. Un calcul direct montre que $8\omega = ldG - 3Gdl$ o\`u ${\displaystyle G = x_1^2 - x_2^3 + \frac{3}{8}bx_1^2x_2 + \frac{3}{8} x_1^3}$; ainsi ${\displaystyle \frac{\omega}{lG}}$ est ferm\'ee, ce qui prouve le th\'eor\`eme dans ce second cas.
\end{proof}

Le Th\'eor\`eme \ref{theo31}, combin\'e \`a la Proposition \ref{prop25}, nous permet d'am\'eliorer le Th\'eor\`eme \ref{theo30} comme il suit. Soit $[{\mathcal F}_{\omega_3}]$ un feuilletage de degr\'e $2$ sur ${\mathbb P}^2_{\mathbb C}$; nous dirons que $[{\mathcal F}_{\omega_3}]$ satisfait la propri\'et\'e ${\mathcal P}$ lorsque:

\begin{description}
\item [${\mathcal P}_1$] Pour tout point singulier $m \in [{\mathcal F}_{\omega_3}]$, si $\theta$ est une $1-$forme d\'efinissant $[{\mathcal F}_{\omega_3}]_{,m}$ alors le $1-$jet de $\theta$ en $m$ est non nul.
\item [${\mathcal P}_2$] Si $j^1\theta_{,m}$ est nilpotent, i.e. conjugu\'e \`a $x_1dx_1$, alors $\mu([{\mathcal F}_{\omega_3}];m) = 2$.
\end{description}

\begin{remark}
La condition ${\mathcal P}$ est satisfaite pour un ouvert de Zariski de feuilletages de degr\'e deux.
\end{remark}

\begin{remark}
Si ${\mathcal P}$ n'est pas satisfaite, alors ou bien il existe un point $m$ tel que, avec les notations pr\'ec\'edentes, $j^1\theta_{,m} = 0$, ou bien il y a un point nilpotent $m$ en lequel $\mu([{\mathcal F}_{\omega_3}];m) \geq 3$. Nous \'etudions ces deux situations plus loin.
\end{remark}

\begin{theorem}
Soit $\omega = \omega_3 + \cdots \in \Omega^1({\mathbb C}^3,0)$ un germe de $1-$forme int\'egrable dont la partie homog\`ene $\omega_3$ est dicritique et ${\rm Cod\ Sing}\ \omega_3\geq 2$. Si $[{\mathcal F}_{\omega_3}]$ satisfait la propri\'et\'e ${\mathcal P}$ on est dans l'une des deux situations suivantes (non exclusives):
\begin{itemize}
\item ${\mathcal F}_\omega$ est d\'efini par une $1-$forme ferm\'ee m\'eromorphe.
\item ${\mathcal F}_\omega$ est holomorphiquement conjugu\'e \`a ${\mathcal F}_{\omega_3}$.
\end{itemize}
\end{theorem}

\begin{proof} Si $[{\mathcal F}_{\omega_3}]$ est sans nilpotence on invoque le Th\'eor\`eme \ref{theo30}. Sinon $[{\mathcal F}_{\omega_3}]$ a un point singulier $m$ nilpotent non trivial en lequel $\mu([{\mathcal F}_{\omega_3}];m) = 2$. Si ${\mathcal F}_\omega$ n'est pas conjugu\'e \`a ${\mathcal F}_{\omega_3}$ c'est qu'on ne peut appliquer le Lemme \ref{lemme21} en un certain point singulier (pour conclure comme dans la Proposition \ref{prop22}); si c'est en $m$, d'apr\`es les Propositions \ref{prop24} et \ref{prop25} le germe $[{\mathcal F}_{\omega_3}]_{,m}$ poss\`ede une int\'egrale premi\`ere holomorphe non constante. D'apr\`es le Th\'eor\`eme \ref{theo31}, $[{\mathcal F}_{\omega_3}]$ est d\'efini par une $1-$forme ferm\'ee rationnelle ${\underline \eta}$. Si c'est en un point $m'$ o\`u le $1-$jet est non nilpotent, alors $m'$ est central et $[{\mathcal F}_{\omega_3}]$ est encore d\'efini par une $1-$forme ferm\'ee. \end{proof}

\begin{theorem}\label{theo32'}
Soit $\omega = \omega_3 + \cdots \in \Omega^1({\mathbb C}^3,0)$ un germe de $1-$forme int\'egrable dont la partie homog\`ene $\omega_3$ est dicritique et ${\rm Cod\ Sing\ }\omega_3 \geq2$. S'il existe un point $m \in {\rm Sing}\ [{\mathcal F}_{\omega_3}]$ tel que $j^1\theta_{,m} \equiv0$ ($\theta$ d\'efinissant $[{\mathcal F}_{\omega_3}]$ en $m$) alors ${\mathcal F}_\omega$ est transversalement affine, i.e. il existe une $1-$forme m\'eromorphe ferm\'ee $\omega_1$ telle que $d\omega = \omega \wedge \omega_1$.
\end{theorem}

\begin{remark}
Il y a une g\'en\'eralisation du Th\'eor\`eme \ref{theo32'} en tout ordre, i.e. en rempla\c cant $\omega_3$ par $\omega_\nu$ et la condition $j^1\theta_{,m}\equiv0$ par $j^{\nu-2}\theta_{,m}\equiv0$. La conclusion et la preuve sont les m\^emes.
\end{remark}

\begin{remark}
D'un point de vue formel, la fonction multivalu\'ee $F = {\rm exp}-\int \omega_1$ est un facteur int\'egrant de $\omega$: $d\big(\frac{\omega}{f}\big) = 0$.
\end{remark}

\noindent {\em D\'emonstration du Th\'eor\`eme \ref{theo32'}.} On se place dans la carte affine $x_3 = 1$ et l'on supose que le $1-$jet de $\zeta = \omega_3\big|_{x_3=1}$ est nul au point $(0:0:1)$, i.e.: $$\zeta = \zeta_2 + q(x_1dx_2 -x_2dx_1)$$ o\`u $\zeta_2$ est une $1-$forme \`a coefficients homog\`enes de degr\'e $2$ non colin\'eaire \`a $x_1dx_2 - x_2dx_1$. Un calcul \'el\'ementaire (par \'eclatement par exemple) montre qu'il existe une $1-$forme m\'eromorphe ferm\'ee ${\underline \eta}$ telle que $d\zeta = \zeta \wedge {\underline \eta}$. Comme dans la preuve du Th\'eor\`eme \ref{theo30}, en utilisant le Th\'eor\`eme de Frobenius et le Th\'eor\`eme d'Extension de L\'evi, on construit une $1-$forme m\'eromorphe ferm\'ee ${\tilde \eta}$ d\'efinie au voisinage de $E^{-1}(0)$ telle que $dE^*\omega = E^*\omega \wedge \tilde{\eta}$. Toujours d'apr\`es L\'evi il existe une $1-$forme m\'eromorphe ferm\'ee $\omega$ telle que $E^*\omega_1 = {\tilde \eta}$. Cet $\omega_1$ convient. \qed

\vskip0.2cm
\noindent {\bf Conjecture.} {\em Soit $[{\mathcal F}_{\omega_3}]$ un feuilletage de degr\'e deux sur ${\mathbb P}^2_{\mathbb C}$. Alors si $\omega = \omega_3 + \cdots \in \Omega^1({\mathbb C}^3,0)$ est int\'egrable on a l'alternative:}
\vskip0.2cm
\begin{minipage}{13cm}
\begin{enumerate}[$(1)$]
\item  {\em ${\mathcal F}_{\omega}$ et ${\mathcal F}_{\omega_3}$ sont conjugu\'es holomorphiquement.}
\item  {\em ${\mathcal F}$ est d\'efini par une $1-$forme ferm\'ee ou bien est transversalement affine.}
\end{enumerate}
\end{minipage}
\vskip0.2cm

Dans les paragraphes qui suivent on traite plusieurs nouveaux cas allant dans le sens de la Conjecture.

\section{Vers la conjecture: le cas $\mu = 3$}

Nous gardons la notation du paragraphe pr\'ec\'edent et nous proposons quelques r\'esultats en pr\'esence d'un point nilpotent $m$ tel que $\mu([{\mathcal F}_{\omega_3}];m) \geq 3$. Comme on l'a dit pr\'ec\'edemment un feuilletage de degr\'e deux a au plus $7$ points singuliers compt\'es avec multiplicit\'e, i.e. $\mu([{\mathcal F}_{\omega_3}];m) \leq 7$ et lorsque l'on classifie, parmi ces feuilletages, ceux qui ne pr\'esentent qu'un point singulier (i.e. $\mu([{\mathcal F}_{\omega_3}];m) = 7$) on constate que le $1-$jet de $[{\mathcal F}_{\omega_3}]_{,m}$ est non nilpotent \cite{Cer-Des}; de sorte que si $m$ est une singularit\'e nilpotente de $[{\mathcal F}_{\omega_3}]$ on a $\mu([{\mathcal F}_{\omega_3}];m) \in \{2,\ldots,6\}$.

En vertu du Th\'eor\`eme \ref{theo26} nous allons nous int\'eresser au feuilletages de degr\'e deux $[{\mathcal F}_{\omega_3}]$ poss\'edant un point singulier nilpotent $m$ en lequel $[{\mathcal F}_{\omega_3}]_{,m}$ poss\`ede un facteur int\'egrant formel. Pour cel\`a nous avons besoin de r\'esultats pr\'eparatoires \'el\'ementaires.

Soit ${\mathcal F}$ un feuilletage de degr\'e $2$ donn\'e dans la carte affine $\{(x_1,x_2)\} = {\mathbb C}^2 \subset {\mathbb P}_{\mathbb C}^2$ par la $1-$forme polynomiale: $$\theta = (x_1 + A_2 - qx_2) dx_1 + (B_2 + qx_1)dx_2 $$ o\`u les $A_2$, $B_2$ et $q$ sont des polyn\^omes homog\`enes de degr\'e $2$. 

\begin{lemma}\label{lemme33}
Si $\mu({\mathcal F};0) \geq 3$ alors la droite $x_1=0$ est invariante.
\end{lemma}

\begin{proof} Si ce n'est pas le cas alors $B_2(0,x_2)$ est non nul et $\mu({\mathcal F};0) = 2$. \end{proof}

Nous allons adopter les notations suivantes:
$$ A_2 = \alpha x_1^2 + \beta x_1 x_2 + \gamma x_2^2,\quad B_2 = x_1(ax_1 + bx_2),\quad q = Px_1^2 + Qx_1x_2 + Rx_2^2.$$
En restriction \`a $x_1=0$ on a $\theta = (\gamma x_2^2 - Rx_2^3) dx_1$; notons que $\gamma$ et $R$ ne peuvent \^etre simultan\'ement nuls. Si $\gamma$ est non nul, ${\mathcal F}$ a deux singularit\'es distinctes sur la droite $x_1=0$ (et peut-\^etre d'autres ailleurs); dans ce cas on peut supposer, \`a automorphisme de ${\mathbb P}_{\mathbb C}^2$ pr\`es, que $\gamma=1$ et $R=0$. Apr\`es cette normalisation les singularit\'es de ${\mathcal F}$ sur $x_1=0$ sont en $0$ et \`a {\em l'infini}. \`A l'inverse si $\gamma$ est nul, ${\mathcal F}$ a une seule singularit\'e sur $x_1=0$ qui est bien s\^ur le point $0$. Dans ce cas on peut supposer que $R=1$.

Nous r\'esumons ces faits dans le

\begin{lemma}\label{lemme34}
Si $\mu({\mathcal F};0)\geq 3$ alors \`a automorphisme de ${\mathbb P}_{\mathbb C}^2$ pr\`es $\theta$ appartient \`a l'une des deux familles:
$$\Omega_1:=\{\theta_1 = x_1[(1+\alpha x_1 + \beta x_2) dx_1 + (ax_1 + bx_2) dx_2] + (Px_1^2 + Qx_1x_2 + x_2^2)(x_1dx_2 - x_2dx_1)\} $$
$$\Omega_2:=\{\theta_2 = x_1[(1+\alpha x_1 + \beta x_2) dx_1 + (ax_1 + bx_2) dx_2] + x_2^2dx_1 +(Px_1^2 + Qx_1x_2)(x_1dx_2 - x_2dx_1)\}.$$
\end{lemma}

\begin{remark}\label{remarque35} On note que pour la premi\`ere famille $\Omega_1$ on a $\mu({\mathcal F}_{\theta_1};0)\geq4$. Dans la seconde famille $\Omega_2$ on a $\mu({\mathcal F}_{\theta_2};0)=3$ si et seulement si le coefficient $b$ est non nul. Si $b=0$ et $a\neq0$ on a $\mu({\mathcal F}_{\theta_2};0)=4$; si $a=b=0$ et $Q\neq0$ alors $\mu({\mathcal F}_{\theta_2};0)=5$ et enfin si $a=b=Q=0$ alors $\mu({\mathcal F}_{\theta_2};0)=6$.
\end{remark}

\noindent {\em Exemples.} La $1-$forme $\theta_1 = (x_1 - x_2^3) dx_1 + x_1x_2^2dx_2$ fait partie de la premi\`ere famille $\Omega_1$; on a $\mu({\mathcal F}_{\theta_1};0)=5$ et ${\mathcal F}_{\theta_1}$ poss\`ede l'int\'egrale premi\`ere rationnelle $ \frac{1}{2x_1^2} - \frac{1}{3} \frac{x_2^3}{x_1^3}$. Dans la seconde famille $\Omega_2$ on trouve par exemple (\`a homoth\'etie pr\`es) $\theta_2 = (x_1 + x_2^2)dx_1 + x_1d(x_1 + x_2^2) = (2x_1 + x_2^2)dx_1 + 2x_1x_2dx_2$ qui est ferm\'ee avec int\'egrale premi\`ere $x_1(x_1+x_2^2)$. Ici on a $\mu({\mathcal F}_{\theta_2};0)=3$ et $b=2$.

\vskip0.2cm
G\'en\'eriquement sur les param\`etres $(\alpha,\beta,a,b,P,Q)$ un feuilletage ${\mathcal F}_{\theta_2}$ induit par un \'el\'ement $\theta_2 \in \Omega_2$ poss\`ede \`a l'origine deux germes de courbes invariantes distinctes: l'une est l'axe $x_1=0$ et l'autre est lisse et tangente \`a $x_1=0$. Pour des valeurs sp\'eciales des param\`etres ceci n'est pas toujours le cas. Pour voir ceci voici comment l'on proc\`ede; le germe lin\'eaire $(x_1 + x_2) dx_1 - x_1dx_2$ poss\`ede un seul germe de courbe invariante \`a l'origine, l'axe $x_1=0$. Son image r\'eciproque $\theta_3$ par le rev\^etement ramifi\'e $(x_1,x_2^2)$:
$$\theta_3 = (x_1 + x_2^2)dx_1 - 2x_1x_2dx_2 $$ poss\`ede la m\^eme propri\'et\'e. La $1-$forme $\theta_3$ est un \'el\'ement de $\Omega_2$, avec $b=-2$ et tous les autres param\`etres nuls. Le feuilletage ${\mathcal F}_{\theta_3}$ est d\'efini par la $1-$forme ferm\'ee ${\displaystyle \frac{\theta_3}{x_1^3}}$, ce qui fait appara\^itre le facteur int\'egrant $x_1^3$ qui poss\`ede donc de la multiplicit\'e. Ce facteur int\'egrant est unique \`a constante multiplicative pr\`es.

Soit $\chi \subset {\mathbb Q}$ l'ensemble: $$\chi:= \left\{ \frac{l-2}{k-1}, l\in {\mathbb N}, k \in {\mathbb N} - \{1\} \right\}  = {\mathbb Q}_{>0} \cup \{-{\mathbb N}\} \cup \left\{ \frac{-1}{{\mathbb N}} \right\} \cup \left\{ \frac{-2}{{\mathbb N}} \right\}$$ avec des notations \'evidentes. \`A l'inverse de cet exemple on a sous une hypoth\`ese de non r\'esonance le lemme technique suivant:

\begin{lemma}\label{lemme37}
Soit $\theta_2$ un \'el\'ement de $\Omega_2$. On suppose que $\theta_2$ poss\`ede un facteur int\'egrant formel \break $f \in \widehat{{\mathcal O}({\mathbb C}^2,0)}$: $d(\frac{\theta_2}{f} )= 0$. Si $b \notin \chi$, alors \`a constante multiplicative non nulle pr\`es $f = x_1 g$ o\`u $g \in \widehat{{\mathcal O}({\mathbb C}^2,0)}$ est une submersion formelle satisfaisant $ \frac{\partial g}{\partial x_1}(0) \neq 0$ et $g(0,x_2) = x_2^2$.
\end{lemma}

\begin{proof} La condition $d(\frac{\theta_2}{f} )= 0$ s'exprime comme suit: 
\begin{equation} \label{eq3} fd\theta_2 + \theta_2 \wedge df = 0. \end{equation} Remarquons que si $f$ est constant alors $\theta_2$ est ferm\'ee; ceci implique que $b=2$, ce qui est interdit par hypoth\`ese. On \'ecrit alors $f = x_1^k g(x_1,x_2)$ avec $k \in {\mathbb N}$ et $g \in \widehat{{\mathcal O}({\mathbb C}^2,0)}$ une s\'erie formelle telle que $g(0,x_2)\not\equiv0$. 

En d\'evelopant (\ref{eq3}) et en faisant $x_1=0$ apr\`es simplification par $x_1^k$ on obtient:
\begin{equation}\label{eq4} g(0,x_2) (b(1-k) - 2) + x_2 \frac{\partial g}{\partial x_2}(0,x_2) \equiv 0.   \end{equation} Quitte \`a multiplier $f$ par une constante, on peut supposer que $$g(0,x_2) = x_2^l + \cdots\ \mbox{ avec}\ l\in {\mathbb N}. $$ On obtient alors d'apr\`es (\ref{eq4}) $$b(1-k) - 2 + l = 0 $$ qui conduit, puisque $b \notin \chi$, \`a $k=1$ et $l=2$.

Puisque maintenant $k=1$ la solution g\'en\'erale de (\ref{eq4}) est $\{\lambda x_2^2,\lambda \in {\mathbb C}\}$, si bien que $g(0,x_2) \equiv x_2^2$. On \'ecrit le d\'eveloppement en s\'erie de $g$ sous la forme suivante: $$g = x_2^2 + x_1\varphi_1(x_2) + x_1^2 \varphi_2(x_2) + \cdots + x_1^i \varphi_i(x_2) + \cdots $$ o\`u les $\varphi_i$ sont des s\'eries formelles en $x_2$. Un calcul \'el\'ementaire montre que (\ref{eq3}) implique:
\begin{equation}\label{eq5} x_2^2 ((a-\beta) + 5Qx_2) + 2x_2(1 + \beta x_2 - Qx_2^2) + x_2^2 \varphi'_1(x_2) - (b+2)x_2 \varphi_1(x_2) = 0.   \end{equation} En particulier le calcul du coefficient de $x_2$ dans (\ref{eq5}) conduit \`a $$2 - (b+2)\varphi_1(0) = 0 $$ et par suite $\varphi_1(0) = \frac{\partial g}{\partial x_1}(0)$ est non nul.
\end{proof}

\begin{remark}
Il y a une possibilit\'e de preuve directe en utilisant que pour $b \notin {\mathbb Q}$ la r\'eduction des singularit\'es en $0$ est celle de $(x_1 + x_2^2)dx_1 + bx_1x_2dx_2$.
\end{remark}

Nous \'enon\c cons maintenant deux r\'esultats voisins qui vont se substituer dans certains cas au Th\'eo\-r\`e\-me de Dulac. Comme nous n'utiliserons pas explicitement le premier nous n'en donnons pas la preuve que nous avons v\'erifi\'ee au prix d'un calcul long et lourd ne mettant en jeu que des arguments utilis\'es dans le Th\'eor\`eme \ref{theo31}.

\begin{theorem}\label{theo38}
Soit ${\mathcal F}$ un feuilletage de degr\'e deux poss\'edant une singularit\'e nilpotente en $0$ telle que $\mu({\mathcal F};0) = 3$. Si ${\mathcal F}_{,0}$ poss\`ede un facteur int\'egrant formel r\'eduit du type $f = x_1 g(x_1,x_2)$, $g = x_1 + \cdots$ alors ${\mathcal F}$ est d\'efini par une $1-$forme ferm\'ee rationnelle.
\end{theorem}

Comme on l'a vu si ${\mathcal F}$ a une singularit\'e nilpotente telle que $\mu({\mathcal F};0) = 3$, alors ${\mathcal F}$ est d\'ecrit, \`a isomorphisme pr\`es, par une $1-$forme $\theta_2 \in \Omega_2$ et le nombre $b = b({\mathcal F})$ est bien d\'efini. On a le:

\begin{theorem}\label{theo39}
Soit ${\mathcal F}$ un feuilletage de degr\'e deux poss\'edant une singularit\'e nilpotente en $0$ telle que $\mu({\mathcal F};0) = 3$. On suppose que ${\mathcal F}_{,0}$ poss\`ede un facteur int\'egrant formel. Si $b = b({\mathcal F}) \notin {\mathbb Q}$ alors ${\mathcal F}$ est d\'efini par une $1-$forme ferm\'ee rationnelle.
\end{theorem}

Le Lemme \ref{lemme37} fait que le Th\'eor\`eme \ref{theo39} est un corollaire de Th\'eor\`eme \ref{theo38}.
\vskip0.2cm
\noindent {\em D\'emonstration du Th\'eor\`eme \ref{theo39}.} Au point $0$ le feuilletage ${\mathcal F}$ est d\'efini par une $1-$forme de type $\theta_2\in \Omega_2$; d'apr\`es le Lemme \ref{lemme37} le facteur int\'egrant formel est de type $f = x_1g$, $g(0,x_2) = x_2^2$, $\frac{\partial g}{\partial x_1}(0) \neq 0$. Comme $\frac{\theta_2}{f}$ est ferm\'ee, il existe des nombres complexes $\lambda$ et $\mu$ et $V \in \widehat{{\mathcal O}({\mathbb C}^2,0)}$ tels que (voir \cite{Cer-Ma} par exemple): $$\frac{\theta_2}{f} = \frac{\theta_2}{x_1g} = \lambda \frac{dx_1}{x_1} + \mu \frac{dg}{g} + dV. $$ Un calcul \'el\'ementaire (direct ou en utilisant la r\'eduction des singularit\'es de ${\mathcal F}_{,0}$) montre que $\frac{\mu}{\lambda} = \frac{b}{2}$. En r\'esulte que si $b$ est non r\'eel positif ou un irrationnel mal approch\'e par les rationnels (conditions de Siegel-Brujno), alors $f$ est holomorphe en $0$ (voir \cite{Cer-Ma} par exemple): c'est une cons\'equence des th\'eor\`emes de Poincar\'e-Siegel appliqu\'es en un point convenable de la r\'eduction des singularit\'es de ${\mathcal F}_{,0}$.

Au point singulier \`a l'infini de $x_1=0$, le feuilletage est d\'efini par une $1-$forme $\theta'$ dont le $1-$jet est:
$$j^1\theta'= (x_3 - Qx_1)dx_1 - (b+1)x_1dx_3. $$ Sous les m\^emes conditions diophantiennes sur $b$, $\theta'$ est holomorphiquement lin\'earisable et poss\`ede donc un facteur int\'egrant holomorphe.

Maintenant, on observe que le probl\`eme est invariant sous l'action d'automorphismes du corps ${\mathbb C}$ (existence d'un facteur int\'egrant formel, feuilletage d\'efini par une forme ferm\'ee rationnelle); si le nombre $b$, qui est non rationnel, est alg\'ebrique il satisfait aux conditions diophantiennes exig\'ees (c'est un r\'esultat de Liouville). Si $b$ est transcendant on peut via l'action d'un automorphisme $\sigma$ se ramener au cas o\`u encore ces conditions sont satisfaites.

Remarquons que l'holonomie de la vari\'et\'e invariante $x_1=0$ du feuilletage d\'efini par $x_3dx_1 - (b+1)x_1dx_3$ n'est pas d'ordre fini si $b\notin {\mathbb Q}$; il en est donc de m\^eme pour l'holonomie de la vari\'et\'e invariante $\big((x_1=0) - \{0\}\big)$ pour ${\mathcal F}$. Consid\'erons maintenant $\theta_2$ vue comme $1-$forme rationnelle globale sur ${\mathbb P}_{\mathbb C}^2$. On dispose en $0$ d'un facteur int\'egrant holomorphe $f$: $d(\frac{\theta_2}{f}) = 0$. Comme ${\mathcal F}$ n'a pas de singularit\'es sur $\Sigma = {\mathbb C}^2 \cap \big((x_1=0) - \{0\}\big)$ qui est invariante on peut prolonger $f$ holomorphiquement le long de $\Sigma$. Au point \`a l'infini $\infty \in \overline{\Sigma} - \Sigma$, on dispose cette fois d'un facteur int\'egrant m\'eromorphe $f_\infty$ h\'erit\'e de la lin\'earisation de $\theta'$. Si $\Delta$ est un petit disque transverse \`a $\Sigma$ au point $m_0 \in \Sigma$, $m_0$ proche de $\infty$, les $1-$formes m\'eromorphes ferm\'ees ${\displaystyle \frac{\theta_2}{f}\Big|_{\Delta}}$ et ${\displaystyle \frac{\theta_2}{f_\infty}\Big|_{\Delta}}$ sont invariantes par l'holonomie $h$ de la vari\'et\'e invariante $\Sigma$ r\'epr\'esent\'ee sur la transversale $\Delta$. Comme $h$ n'est pas d'ordre fini ${\displaystyle \frac{\theta_2}{f}\Big|_{\Delta}}$ et ${\displaystyle \frac{\theta_2}{f_\infty}\Big|_{\Delta}}$ sont ${\mathbb C}-$colin\'eaires. On peut donc, quitte \`a multiplier $f_\infty$ par une constante, supposer que: $$ \frac{\theta_2}{f}\Big|_{\Delta} = \frac{\theta_2}{f_\infty}\Big|_{\Delta}.$$
Par suite au voisinage de $x_1=0$ dans ${\mathbb P}_{\mathbb C}^2$ le feuilletage ${\mathcal F}$ est d\'efini par une $1-$forme m\'eromorphe ferm\'ee. Cette $1-$forme s'\'etend en une $1-$forme m\'eromorphe globale sur ${\mathbb P}_{\mathbb C}^2$, qui est donc rationnelle. \qed

Avec les notations pr\'ec\'edentes on a le:

\begin{corollary}\label{cor40} Soit $\omega = \omega_3 + \cdots \in \Omega^1({\mathbb C}^3,0)$ un germe de $1-$forme int\'egrable \`a l'origine de ${\mathbb C}^3$ dont la partie homog\`ene $\omega_3$ est dicritique et ${\rm Cod\ Sing}\ \omega_3 \geq 2$. On suppose que le feuilletage de degr\'e deux $[{\mathcal F}_{\omega_3}]$ poss\`ede un point singulier $m \in {\mathbb P}_{\mathbb C}^2$ tel que $\mu([{\mathcal F}_{\omega_3}];m)=3$; si $b = b([{\mathcal F}_{\omega_3}])$ est non rationnel on a l'alternative:
\begin{enumerate}[$(1)$]
\item ${\mathcal F}_\omega$ est d\'efini par une $1-$forme ferm\'ee m\'eromorphe.
\item ${\mathcal F}_\omega$ est holomorphiquement conjugu\'e \`a ${\mathcal F}_{\omega_3}$.
\end{enumerate}
\end{corollary}

\begin{proof}
Commen\c cons par quelques pr\'ecisions. On peut supposer que $[{\mathcal F}_{\omega_3}]$ est donn\'e en carte affine par:
$$\theta_2 = [x_1(1+\alpha x_1 + \beta x_2) + x_2^2 - x_2(Px_1^2 + Qx_1x_2)]dx_1 + x_1(ax_1 + bx_2 + Px_1^2 + Qx_1x_2)dx_2. $$ Au point singulier $m'$ \`a l'infini de $x_1=0$, le feuilletage est donn\'ee par $\theta'\in \Omega^1({\mathbb P}_{\mathbb C}^2,m')$, avec pour $1-$jet $j^1\theta'= (x_3 - Qx_1)dx_1 - (b+1)x_1dx_3$. En particulier $\mu([{\mathcal F}_{\omega'}];m')$ est plus grand ou \'egal que $1$ (et m\^eme vaut $1$ si $b$ est diff\'erent de $-1$). Par suite si $m''$ est un point singulier distinct de $m$ et $m'$ alors $\mu([{\mathcal F}_{\omega_3}];m'')\leq 3$. En r\'esulte que si $\theta''\in \Omega^1({\mathbb P}_{\mathbb C}^2,m'')$ est une $1-$forme d\'efinissant $[{\mathcal F}_{\omega_3}]_{,m''}$ alors le $1-$jet $j^1\theta''_{,m''}$ est non identiquement nul. Supposons que $j^1\theta''_{,m''}$ soit nilpotent avec $\mu([{\mathcal F}_{\omega_3}];m'') = 3$. 

Comme $m''$ ne peut \^etre situ\'e sur $x_1=0$ et $[{\mathcal F}_{\omega_3}]_{,m}$ ne poss\`ede qu'une droite invariante, le feuilletage $[{\mathcal F}_{\omega_3}]$ poss\`ede une seconde droite invariante $L$ passant par $m''$ et coupant $x_1=0$ n\'ecessairement en $m'$; c'est une cons\'equence du Lemme \ref{lemme33}. Par suite $L$ est ou bien une droite $x_1 = \varepsilon$ ou bien la droite \`a l'infini $x_3=0$. Un calcul \'el\'ementaire montre que, quitte \`a conjuguer par une transformation de type $ \left( \frac{x_1}{x_1 - \varepsilon},i\varepsilon^{1/2}\frac{x_2}{x_1 - \varepsilon}\right)$, on peut supposer que $L$ est la droite \`a l'infini $x_3=0$, cette transformation laissant invariant l'espace $\Omega_2$ et le coefficient $b$. Ainsi $\theta_2$ est du type: $$ \theta_2 = (x_1(1+\alpha x_1 + \beta x_2) + x_2^2)dx_1 + x_1(ax_1+bx_2)dx_2 .$$ 

Sous nos hypoth\`eses ${\mathcal F}_{\theta_2}$ a pour singularit\'es les trois points $m$, $m'$, $m''$. Un second calcul \'el\'ementaire montre que ${{\mathcal F}_{\theta_2,m''}}$ a une singularit\'e nilpotente \`a nombre de Milnor $3$ si et seulement si $\alpha = \beta = a = 0$. Mais dans ce cas ${\mathcal F}_{\theta_2}$ est donn\'e par une forme ferm\'ee m\'eromorphe; en r\'esulte que le feuilletage ${\mathcal F}_\omega$ aussi.

Dans la suite on peut donc supposer que le point $m$ est le seul point nilpotent \`a nombre de Milnor $3$; les autres points singuliers ont donc leur nombre de Milnor $\leq2$. Si ${\mathcal F}_\omega$ n'est pas conjugu\'e \`a ${\mathcal F}_{\omega_3}$ c'est que l'on ne peut appliquer le Lemme \ref{lemme21}. En utilisant les Th\'eor\`emes \ref{theo26}, \ref{theo29}, \ref{theo31}, \ref{theo39} et la Proposition \ref{prop24} suivant la nature des points singuliers, on constate que $[{\mathcal F}_{\omega_3}]$ est donn\'e par une $1-$forme ferm\'ee m\'eromorphe, et donc ${\mathcal F}_\omega$ aussi.
\end{proof}

\section{Le cas $\mu=4$}

Soit ${\mathcal F}$ un feuilletage de degr\'e $2$ ayant une singularit\'e nilpotente en $0$; on suppose que $\mu({\mathcal F};0)=4$. Alors ${\mathcal F}$ appartient, \`a conjugaison pr\`es, \`a la famille $\Omega_1$, i.e. est donn\'e par une $1-$forme de type $\theta_1$ avec $b\neq 0$, ou bien \`a la famille $\Omega_2$ avec $b=0$ et $a\neq 0$. Nous dirons suivant le cas que ${\mathcal F}$ {\em est de type} $\Omega_1$ ou $\Omega_2$. Si ${\mathcal F}$ est de type $\Omega_2$, nous dirons que ${\mathcal F}$ est {\em non radial \`a l'infini} si $Q = Q(\theta_2) = Q({\mathcal F})$ est non nul.

\begin{proposition}\label{prop7}
Soit ${\mathcal F}$ un feuilletage de degr\'e deux sur ${\mathbb P}_{\mathcal C}^2$ de type $\Omega_2$; on suppose que $b({\mathcal F}) = Q({\mathcal F}) = 0$, i.e. $\mu({\mathcal F};0)\geq 4$ et ${\mathcal F}$ est radial \`a l'infini. Alors ${\mathcal F}$ est transversalement projectif.
\end{proposition}

\begin{proof} ${\mathcal F}$ est d\'efini au point $(0:1:0)$ par une $1-$forme dont le $1-$jet est $x_1dx_3 - x_3dx_1$; tout feuilletage de degr\'e $2$ ayant une singularit\'e de type radial est transversalement projectif (\'equation de Ricatti).
\end{proof}

Nous allons d\'emontrer le r\'esultat qui suit.

\begin{theorem}\label{theo41}
Soit $\omega = \omega_3 + \cdots \in \Omega^1({\mathbb C}^3,0)$ int\'egrable, $\omega_3$ dicritique, ${\rm Cod\ Sing}\ \omega_3 \geq2$. On suppose que $[{\mathcal F}_{\omega_3}]$ poss\`ede une singularit\'e nilpotente $m$ avec $\mu([{\mathcal F}_{\omega_3}];m)=4$ de type $\Omega_i$. Alors ${\mathcal F}_\omega$ est conjugu\'e \`a ${\mathcal F}_{\omega_3}$.
\end{theorem}

\noindent {\em D\'emonstration.} Dans la suite on suppose $m = (0:0:1)$ et que $[{\mathcal F}_{\omega_3}]$ est donn\'e par une forme de type $\theta_i$. Dans les deux cas la r\'eduction des singularit\'es du point nilpotent se fait en deux \'eclatements. On note $Ex(1)$ le premier diviseur apparaissant au premier \'eclatement, $Ex(2)$ le second; sur $Ex(1)$ il y a un seul point singulier $p_3 = Ex(1) \cap Ex(2)$. Pour les deux configurations $\Omega_1$, $\Omega_2$, le point $p_3$ est une singularit\'e {\em r\'esonnante} \cite{Mar-Ra} avec $1-$jet de type $xdy + 2ydx$, $(y=0) = Ex(1)_{,p_3}$, $(x=0) = Ex(2)_{,p_3}$. Comme sur $Ex(1)$ il n'y a qu'une singularit\'e, l'holonomie de la vari\'et\'e invariante locale $(y=0)$ est triviale et d'apr\`es Mattei-Moussu \cite{Ma-Mo} il y a une int\'egrale premi\`ere de type $x^2y$ en $p_3$. Par suite l'holonomie de la vari\'et\'e invariante locale $Ex(2)_{,p_3}$ est une involution $x \rightarrow -x$. On note $p_1 \in Ex(2)$ le point singulier o\`u aboutit le transform\'e strict de la droite invariante $x_1=0$ pour $\theta_i$. Enfin il y a un troisi\`eme point singulier, not\'e $p_2$. 

Dans la situation o\`u ${\mathcal F}$ est de type $\Omega_2$, le point singulier $p_1$ est r\'esonnant, tandis que le point $p_2$ est de type noeud-col. En $p_2$ arrive une courbe invariante formelle transverse \`a $Ex(2)$. \`A l'inverse dans le cas o\`u ${\mathcal F}$ est de type $\Omega_1$, le point $p_2$ est r\'esonnant tandis que le point $p_1$ est un noeud-col {\em mal-orient\'e} au sens o\`u la {\em vari\'et\'e faible} invariante en $p_1$ est pr\'ecisement la droite $x_1=0$.  

Concernant la configuration dans ${\mathbb P}_{\mathbb C}^2$ des autres points singuliers, nous affirmons qu'il n'existe pas d'autre point singulier nilpotent \`a nombre de Milnor $3$. En effet s'il existait un tel point $m'$, il serait de type $\Omega_2$ (avec $b\neq0$); $m'$ n'\'etant pas sur $x_1=0$, il y aurait une droite invariante passant par $m'$ et il y aurait donc sur cette droite un autre point singulier $m''$ avec $1-$jet non nilpotent; donc $m''$ est distinct de $m$, ce qui nous donne au moins $8$ points singuliers compt\'es avec multiplicit\'e et est donc absurde.

Par suite les autres points singuliers ont leur nombre de Milnor inf\'erieur ou \'egal a $2$, en particulier en chacun des points singuliers le $1-$jet est non nul. En chaque point singulier $m'$ autre que $m$, on a $\mu([{\mathcal F}_{\omega_3}];m')\leq2$; s'il n'existe pas de champ tangent \`a $E^{-1}({\mathcal F}_\omega)$ en $m'$ et transverse au diviseur $E^{-1}(0)_{,m}$ alors $[{\mathcal F}_{\omega_3}]$ et donc ${\mathcal F}_\omega$ est d\'efini par une forme ferm\'ee m\'eromorphe. Reste \`a \'etudier $E^{-1}({\mathcal F}_\omega)$ au point $m = (0:0:1)$. S'il n'y a pas de champ $X_{,m}$ comme ci-dessus on sait que $[{\mathcal F}_{\omega_3}]_{,m} = {\mathcal F}_{\theta_i}$ poss\`ede un facteur int\'egrant formel $f$: ${\displaystyle d\big(\frac{\theta_2}{f}\big)=0}$. 

Pour terminer la d\'emonstration du th\'eor\`eme il nous suffit d'\'etablir le

\begin{lemma}
Si $\mu({\mathcal F}_{\theta_i};0)=4$, avec $\theta_i \in \Omega_i$, alors $\theta_i$ n'a pas de facteur int\'egrant formel.
\end{lemma}

\begin{proof} Si tel \'etait le cas, l'holonomie du diviseur exceptionnel $Ex(2)$ serait ab\'elienne \cite{Cer-Ma}. Mais pour chaque $\theta_i$ il y a un point $p_i$ en lequel la singularit\'e est de type noeud-col avec nombre de Milnor $2$. En r\'esulte que l'holonomie de la vari\'et\'e invariante locale $Ex(2)_{,p_1}$ est du type \break $\varphi: x \rightarrow x + cx^2 + \cdots$, $c\neq0$. Mais un tel $\varphi$ ne peut commuter \`a une involution $x \rightarrow -x$. \end{proof}

\section{Le cas $\mu=5$}

Soit ${\mathcal F}$ un feuilletage de degr\'e deux \`a singularit\'e nilpotente \`a l'origine de ${\mathbb C}^2$ avec $\mu({\mathcal F};0)=5$. On peut supposer ${\mathcal F}$ donn\'e par une $1-$forme $\theta$ appartenant \`a l'une des familles $\Omega_1$, $\Omega_2$. Si $\theta = \theta_1 \in \Omega_1$ on v\'erifie que $b=0$, i.e. que $$\theta_1 = [x_1(1+\alpha x_1 + \beta x_2) - x_2(Px_1^2 + Qx_1x_2 + x_2^2)]dx_1 + x_1(ax_1 + Px_1^2 + Qx_1x_2 + x_2^2)dx_2. $$ Remarquons que l'on peut faire $P$ ou $Q$ nuls.

Si $\theta = \theta_2 \in \Omega_2$ on a $a=b=0$ et $Q \neq 0$, i.e. apr\`es normalisation $$\theta_2 = [x_1(1+\alpha x_1 + \beta x_2) + x_2^2 - x_2(Px_1^2 + x_1x_2)]dx_1 + x_1^2(Px_1+x_2)dx_2. $$

Sous les hypoth\`eses pr\'ec\'edentes on a la

\begin{proposition}\label{prop8}
Si $\theta_2$ poss\`ede un facteur int\'egrant formel, alors ${\mathcal F}_{\theta_2}$ est d\'efini par une $1-$forme ferm\'ee rationnelle.
\end{proposition}

\begin{proof} La singularit\'e $(0:1:0)$ est de type $(x+y)dx -xdy$. En r\'esulte que l'holonomie de la vari\'et\'e invariante $x_1=0$ est de la forme ${\displaystyle y \rightarrow \frac{y}{1-y}}$. Dans la r\'esolution des singularit\'es de ${\mathcal F}_{\theta_2,0}$, la singularit\'e r\'esonnante $p_1$ est donc holomorphiquement normalisable \cite{Mar-Ra}, sans int\'egrale premi\`ere holomorphe. En particulier en $p_1$ il y a un unique facteur int\'egrant formel, \`a constante multiplicative pr\`es, et celui-ci est en fait convergent. Par suite si $f$ est un facteur int\'egrant formel de $\theta_2$ en $0$, son relev\'e en $p_1$ est convergent; en r\'esulte que $f$ converge. On prolonge alors la forme m\'eromorphe ferm\'ee ${\displaystyle \frac{\theta_2}{f}}$ le long de $x_1=0$, comme on l'a fait en Th\'eor\`eme \ref{theo39}. Ceci est possible car l'holonomie de $x_1=0$ est non p\'eriodique.\end{proof}

De la m\^eme fa\c con on a la:

\begin{proposition}
Si $\theta_1$ poss\`ede un facteur int\'egrant formel $f$ alors ${\mathcal F}_{\theta_1}$ est d\'efini par une $1-$forme ferm\'ee rationnelle.
\end{proposition}

\begin{proof} Si l'on proc\`ede \`a trois \'eclatements successifs (grosso modo, on pose $x_1 = sx_2^3$) on trouve un point singulier dont le $1-$jet est $6udv - vdu$. Les th\'eor\`emes de lin\'earisation et normalisation de Poincar\'e assurent qu'en un tel point tout facteur int\'egrant formel est en fait convergent. Ceci implique que $f$ est convergent; la $1-$forme ferm\'ee m\'eromorphe ${\displaystyle \frac{\theta_1}{f}}$ locale se prolonge le long de $\overline{x_1=0} \subset {\mathbb P}^2_{\mathbb C}$ puis \`a tout ${\mathbb P}^2_{\mathbb C}$.\end{proof}

Consid\'erons un d\'eploiement du germe de feuilletage ${\mathcal F}_{\theta_2}$; il suit des Proposition \ref{prop25} et Remarque \ref{rem11} que si ce d\'eploiement est non trivial et si $\theta_2$ n'a pas de facteur int\'egrant formel alors la forme normale de Loray de $\theta_2$ est de type: $$\theta_L = xdx + (y^4 + xl(y^2))dy^2\ {\rm avec}\ \lambda = l(0); $$ plus pr\'ecis\'ement il existe des coordonn\'ees $x,y$ et une unit\'e $U = 1+u_1x + u_2y + \cdots$ telles que $\theta_2 = U\theta_L$; on constate que $d\theta_2 = (-\beta x_1 - 2x_2 + \cdots) dx_1\wedge dx_2$. En particulier la tangente aux z\'eros de $d\theta_2$ ($2x_2 + \beta x_1=0$) est diff\'erente du c\^one tangent $x_1=0$ de $\theta_2$. Maintenant le c\^one tangent de $\theta_L$ est $x=0$ et celui de $d\theta_L$ est $2\lambda y - u_2x=0$. Pour qu'ils soient distincts il est n\'ecessaire que $\lambda$ soit non nul. La Remarque \ref{rem11} dit que, sous cette hypoth\`ese, $\theta_2$ poss\`ede un facteur int\'egrant formel.

\begin{remark}
La $1-$forme $\theta_2 = (x_1 + x_2^2 - x_1x_2^2)dx_1 + x_1^2x_2dx_2$ satisfait la Proposition \ref{prop8}, i.e. ${\mathcal F}_{\theta_2}$ est d\'efini par une $1-$forme ferm\'ee rationnelle.
\end{remark}

Examinons maintenant le cas o\`u $\theta = \theta_1 \in \Omega_1$; toujours sous l'hypoth\`ese d'un d\'eploiement non trivial la forme normale de Loray de $\theta_1$ est, si $\theta_1$ n'a pas de facteur int\'egrant, encore de la forme: $$\theta_L = xdx + (y^4 + xl(y^2))dy^2. $$ On exploite ici le fait que l'involution $(x,y)\rightarrow (x,-y)$ laisse invariante $\theta_L$. Le feuilletage ${\mathcal F}_{\theta_1}$ est donc lui aussi laiss\'e invariant par une involution $I$ de type $(x,-y)$. On consid\`ere alors le sch\'ema de r\'eduction des singularit\'es de $\theta_1$; le premier \'eclatement produit un premier diviseur $Ex(1)$ contenant un seul point singulier $p_1$ correspondant \`a $x_1=0$. On \'eclate alors $p_1$, ce qui produit un nouveau diviseur $Ex(2)$ et deux points singuliers; le croisement $Ex(1) \cap Ex(2)$ not\'e encore $p_1$ et un point $p_2 \neq p_1$ o\`u aboutit $x_1=0$. En $p_1$ le feuilletage est r\'eduit, mais il faut encore \'eclater $p_2$, ce qui produit un diviseur $Ex(3)$ et trois points singuliers distincts: $p_2 = Ex(2)\cap Ex(3)$, $p_3$ o\`u aboutit $x_1=0$ et un troisi\`eme point singulier $p_4$ distinct de $p_2$ et $p_3$. On rel\`eve l'involution $I$ \`a cette configuration pour obtenir une involution $\widetilde{I}$. La restriction de $\widetilde{I}$ \`a $Ex(3)$ fixe les points $p_2$ et $p_3$. Un calcul \'el\'ementaire montre que $\widetilde{I}\big|_{Ex(3)}$ est de type $s \rightarrow -s$. Remarquons que puisque $I$ laisse invariante le feuilletage ${\mathcal F}_{\theta_2}$, le point $p_4$ doit \^etre laiss\'e fixe par $\widetilde{I}$; ce qui est absurde. Ce cas n'arrive donc pas.

On en d\'eduit en utilisant les m\^emes id\'ees que pr\'ec\'edemment le:

\begin{theorem} Soit $\omega = \omega_3 + \cdots \in \Omega^1({\mathbb C}^3,0)$ int\'egrable, $\omega_3$ dicritique, ${\rm Cod\ Sing}\ \omega_3 \geq 2$. On suppose que $[{\mathcal F}_{\omega_3}]$ poss\`ede un point nilpotent $m$ tel que $\mu([{\mathcal F}_{\omega_3}];m)=5$. Alors ou bien ${\mathcal F}_{\omega}$ et ${\mathcal F}_{\omega_3}$ sont conjugu\'ees ou bien ${\mathcal F}_{\omega}$ est d\'efini par une $1-$forme ferm\'ee m\'eromorphe.
\end{theorem}

\section{Le cas $\mu=6$} 

Soit ${\mathcal F}$ un feuilletage de degr\'e deux sur ${\mathbb P}^2_{\mathbb C}$ ayant une singularit\'e nilpotente $m=(0:0:1)$ avec $\mu({\mathcal F};m)=6$. On peut supposer que ${\mathcal F}$ est d\'efini par une $1-$forme $\theta$ appartenant \`a l'une des familles $\Omega_i$. Le lemme qui suit est \'el\'ementaire:

\begin{lemma}\label{lemme7} Si $\mu({\mathcal F};m)=6$ alors $\theta$ appartient \`a la famille $\Omega_2$; avec les notations habituelles on a de plus $a=b=Q=0$.
\end{lemma}

Ainsi ${\mathcal F}$ est d\'efini par une $1-$forme de type
$$\theta_2 = [x_1(1+\alpha x_1 + \beta x_2) + x_2^2 - Px_1^2x_2]dx_1 + Px_1^3 dx_2 $$ avec $P\neq0$. Quitte \`a faire un automorphisme de ${\mathbb P}^2_{\mathbb C}$ {\em ad-hoc}, on se ram\`ene \`a $P=1$, $\alpha = \beta = 0$. On note $\theta_2^0$ la $1-$forme correspondant. Il y a donc un seul feuilletage de ce type \`a conjugaison pr\`es not\'e ${\mathcal F}_0$.

\begin{proposition}\label{prop9} Le feuilletage ${\mathcal F}_0$ est transversalement projectif.\end{proposition}

\begin{proof} Apr\`es \'eclatement du point \`a l'infini $(0:1:0)$, ${\mathcal F}_0$ est transverse \`a la fibration $x_1= {\rm cte}$ en dehors de la droite $x_1=0$, ce qui implique la proposition.\end{proof}

Le feuilletage ${\mathcal F}_0$ est d\'efini dans la carte $x_1=1$ par $$dx_2 - (x_3+x_2^2)dx_3. $$ Il s'agit de la fameuse \'equation d'Airy utilis\'ee pour les probl\`emes de diffraction. Liouville a d\'emontr\'e que les solutions de cette \'equation diff\'erentielle ne sont pas {\em Liouvilliennes}. Ceci signifie, en utilisant les travaux de Singer \cite{Sin} et Casale \cite{Cas} que ${\mathcal F}_0$ n'est pas transversalement affine: il n'existe pas de $1-$forme ferm\'ee rationnelle $\eta$ telle que $d\theta_2^0 = \theta_2^0 \wedge \eta$. En particulier ${\mathcal F}_0$ n'est pas d\'efini par une $1-$forme ferm\'ee rationnelle.

En fait il y a une version locale de ce qui pr\'ec\`ede.

\begin{lemma}\label{lemme8} La $1-$forme $\theta_2^0$ ne poss\`ede pas de facteur int\'egrant formel.\end{lemma}

\begin{proof} Comme $Q=0$, la singularit\'e $(0:1:0)$ de ${\mathcal F}_0$ est de type radial, i.e. poss\`ede une int\'egrale premi\`ere de la forme ${\displaystyle \frac{y}{x} = {\rm cte}}$. En r\'esulte que l'holonomie de la vari\'et\'e invariante $x_1=0$ est triviale. Dans la r\'esolution des singularit\'es ceci implique, d'apr\`es \cite{Ma-Mo}, qu'en $p_1$ le feuilletage transform\'e de ${\mathcal F}_0$ poss\`ede au point $p_1$ une int\'egrale premi\`ere de type $uv^2 = {\rm cte}$ avec $(v=0) = Ex(2)_{,p_1}$. Ceci implique aussi que, pour un bon choix de coordonn\'ees, l'holonomie de la vari\'et\'e invariante $Ex(2)$ au point $p_1$ est la sym\'etrie $v \rightarrow -v$. On se place maintenant au point $p_2$ o\`u le point singulier est de type noeud-col avec nombre de Milnor $4$: au point $p_2$ le feuilletage est \`a conjugaison pr\`es du type $$(x+\cdots)dy + y^4dx $$ o\`u cette fois $(y=0) = Ex(2)_{,p_2}$. Le diff\'eomorphisme d'holonomie da la vari\'et\'e invariante $Ex(2)_{,p_2}$ en $p_2$ est donc de la forme $y + cy^4 + \cdots$, $c \in {\mathbb C}^*$. Notons qu'un tel diff\'eomorphisme ne peut commuter avec une involution de type $y \rightarrow -y$. Or si $\theta_2^0$ poss\'edait un facteur int\'egrant formel, le groupe d'holonomie du diviseur $Ex(2)$ serait ab\'elien. D'o\`u le lemme.\end{proof}

On en d\'eduit le

\begin{theorem} Soit $\omega = \omega_3 + \cdots \in \Omega^1({\mathbb C}^3,0)$ int\'egrable, $\omega_3$ dicritique, ${\rm Cod\ Sing}\ \omega_3 \geq 2$. On suppose que $[{\mathcal F}_{\omega_3}]$ poss\`ede une singularit\'e nilpotente $m$ avec $\mu([{\mathcal F}_{\omega_3}];m)=6$. Alors ${\mathcal F}_{\omega}$ est conjugu\'e \`a ${\mathcal F}_{\omega_3}$.
\end{theorem}

\begin{proof} C'est une cons\'equence du Lemme \ref{lemme8}.   \end{proof}

\section{Krull-d\'eformations pour certains feuilletages transversalement affines}

On s'int\'eresse ici aux feuilletages ${\mathcal F}_\omega$ o\`u $\omega = \omega_3 + \cdots \in \Omega^1({\mathbb C}^3,0)$ et $[{\mathcal F}_{\omega_3}]$ est transversalement affine comme dans le Th\'eor\`eme \ref{theo32'}. On suppose que $[{\mathcal F}_{\omega_3}]$ est donn\'e en carte affine par $$ \theta_0 = \omega_2 + q(x_1dx_2 - x_2dx_1)$$ o\`u $\omega_2$ est une $1-$forme \`a coefficients homog\`enes de degr\'e deux et $q$ une forme quadratique. Pour un choix g\'en\'erique de $\omega_2$ et $q$, le feuilletage ${\mathcal F}_{\theta_0}$ n'est pas d\'efini par une $1-$forme ferm\'ee. En fait le feuilletage local ${\mathcal F}_{\theta_0}$ ne poss\`ede pas en g\'en\'eral de facteur int\'egrant formel. Nous allons faire les hypoth\`eses (g\'en\'eriques) suppl\'ementaires suivantes:
\vskip0.2cm
\begin{minipage}{14cm}
\begin{enumerate}[(i)]
\item \`A conjugaison lin\'eaire pr\`es ${\displaystyle \omega_2 = x_1x_2(x_1 - x_2)\Big(\lambda_1 \frac{dx_1}{x_1} + \lambda_2 \frac{dx_2}{x_2} + \lambda_3\frac{d(x_2 - x_1)}{x_2 - x_1}\Big)}$, \break $\lambda_1 + \lambda_2 + \lambda_3=1$, condition qui est r\'ealis\'ee d\`es que le c\^one tangent est r\'eduit.

\item $\mu({\mathcal F}_{\theta_0};0)=4$, i.e. $\lambda_1 \lambda_2 \lambda_3 \neq 0$.

\item Aux autres points singuliers $m_i$, que l'on peut supposer \`a distance finie, la condition de Kupka $d\theta(m_i) \neq 0$ est satisfaite.

\item Les $\lambda_i$ sont deux \`a deux distincts.
\item ${\mathcal F}_{\theta_0}$ n'est pas d\'efini par une $1-$forme ferm\'ee.
\end{enumerate}
\end{minipage}
\vskip0.2cm

Un calcul \'el\'ementaire montre que $$d\theta_0 + \theta_0 \wedge \left( (1+\lambda_1)\frac{dx_1}{x_1} + (1 + \lambda_2) \frac{dx_2}{x_2} + (1 + \lambda _3) \frac{d(x_2-x_1)}{x_2-x_1} \right) = 0 $$ ce qui produit d'ailleurs une $1-$forme $\omega_1$ comme dans le Th\'eor\`eme \ref{theo32'}. Notons que puisque ${\mathcal F}_{\theta_0}$ n'est pas d\'efini par une forme ferm\'ee, les $\lambda_i$ ne sont pas dans ${\mathbb Z}$.

Soit $[{\mathcal F}_{\omega_3}]$ d\'efini par $\theta_0$ comme ci-dessus. Nous dirons que $[{\mathcal F}_{\omega_3}]$ est {\em transversalement affine g\'en\'erique (TAG)} si les conditions (i), (ii), (iii), (iv), (v) sont satisfaites.

\begin{theorem} Soit ${\mathcal F}_\omega$ un germe de feuilletage \`a l'origine de ${\mathbb C}^3$ donn\'e par $\omega = \omega_3 +\cdots$, $\omega_3$ dicritique. Si $[{\mathcal F}_{\omega_3}]$ est TAG alors ${\mathcal F}_{\omega}$ et ${\mathcal F}_{\omega_3}$ sont conjugu\'es.
\end{theorem}

\begin{proof} Apr\`es un \'eclatement dans la carte o\`u $E = (x_1x_3,x_2x_3,x_3)$ on a avec les notations pr\'ec\'edentes: $$\theta = \frac{E^*\omega}{x_3^4} = \omega_2 + q(x_1dx_2 - x_2dx_1) + x_3(Adx_1 + Bdx_2) + hdx_3. $$ Comme $\theta\big|_{x_3=0} = \theta_0$, il existe une $1-$forme ferm\'ee m\'eromorphe $\theta_1$ telle que $d\theta + \theta \wedge \theta_1 = 0$, avec $$\theta_1\big|_{x_3=0} = (1+\lambda_1)\frac{dx_1}{x_1} + (1 + \lambda_2) \frac{dx_2}{x_2} + (1 + \lambda _3) \frac{d(x_2-x_1)}{x_2-x_1}. $$ Puisque $\theta_1$ est construite par prolongement de $\theta_1\big|_{x_3=0}$, $\theta_1$ n'a pas de p\^ole le long de $x_3=0$. La condition (iv) implique que $\theta_1$ a ses p\^oles le long de trois surfaces lisses deux \`a deux transverses dont les intersections avec $x_3=0$ sont $x_1x_2(x_2-x_1) = 0$. \`A conjugaison holomorphe locale pr\`es on peut supposer que $\theta_1$ a ses p\^oles le long de: $$x_1=0, \quad x_2 = 0,\quad x_2 - x_1 + z^k = 0 $$ pour un certain $k \in {\mathbb N}$. En r\'esulte que la forme ferm\'ee $\theta_1$ est du type suivant en $(0,0,0)$: $$\theta_1 = (1+\lambda_1)\frac{dx_1}{x_1} + (1 + \lambda_2) \frac{dx_2}{x_2} + (1 + \lambda _3) \frac{d(x_2-x_1 + z^k)}{x_2-x_1+z^k} + df $$ avec $f$ holomorphe en $(0,0,0)$, $f(x_1,x_2,0)\equiv0$.

Soit $U$ une unit\'e, $U(0)=1$, et $\phi$ le diff\'eomorphisme: $$\phi = (x_1U^k, x_2U^k, x_3U). $$ On pose $\theta'_1 = \theta_1 - df$; un calcul direct montre que $$\phi^*\theta'_1 = \theta'_1 + 4k\frac{dU}{U}. $$ Par suite, si $k$ est non nul, on peut supposer que $\theta'_1 = \theta'$ (choisir $U = {\rm exp} \frac{-f}{4k}$). Apr\`es cette transformation on a: $$d\theta + \theta \wedge \left( (1+\lambda_1)\frac{dx_1}{x_1} + (1 + \lambda_2) \frac{dx_2}{x_2} + (1 + \lambda _3) \frac{d(x_2-x_1+x_3^k)}{x_2-x_1+x_3^k} \right) = d\theta + \theta \wedge \theta_1. $$ On consid\`ere alors l'application $\sigma$: $$\sigma = (x_1,x_2,e^{\frac{2\pi i}{k}} x_3). $$ On a $\sigma^*\theta_1 = \theta_1$; on introduit la moyenne $\overline{\theta}$ de $\theta$ sous l'action de $\sigma$:   $$ \overline{\theta} = \theta + \sigma^* \theta + \cdots + \sigma^{k-1 *}\theta$$ qui est bien s\^ur invariante sous l'action de $\sigma$; on a $\overline{\theta}\big|_{x_3=0} = k \theta_0$, ce qui indique que $\overline{\theta}$ est non triviale. L'invariance de $\overline{\theta}$ sous l'action de $\sigma$ dit que $\overline{\theta}$ est image r\'eciproque par l'application $\varphi = (x_1,x_2,x_3^k)$ d'une certaine $1-$forme $\alpha$ v\'erifiant $\alpha\big|_{x_3=0} = k\theta_0$ et: $$d\alpha + \alpha \wedge \left( (1+\lambda_1)\frac{dx_1}{x_1} + (1 + \lambda_2) \frac{dx_2}{x_2} + (1 + \lambda _3) \frac{d(x_2-x_1+x_3)}{x_2-x_1+x_3} \right) = d\alpha + \alpha \wedge \beta_1. $$ On a donc un nouveau d\'eploiement de $\theta_0$ mais avec $x_1x_2(x_2-x_1-x_3)=0$ comme hypersurface invariante.

Comme cette hypersurface est un croisement normal, lorsque l'on \'eclate l'origine par $E$, l'holonomie du diviseur exceptionnel $E^{-1}(0)$ pour ${\mathcal F}_{E^*\alpha}$ est ab\'elienne. En r\'esulte que l'holonomie du diviseur exceptionnel apr\`es un \'eclatement du feuilletage ${\mathcal F}_{\theta_0}$ est aussi ab\'elienne. Ce qui est exclus par la condition (v). Par suite $k=0$ et $\theta$ poss\`ede la surface invariante $x_1x_2(x_2-x_1)=0$. Il est facile de voir que $\theta = F^*\theta_0$ o\`u $F$ est une submersion {\em ad-hoc}. On conclut en invoquant le Lemme \ref{lemme8}. \end{proof}

\section{Perturbations de $1-$formes ferm\'ees: exemples}

On consid\`ere dans ce paragraphe des feuilletages ${\mathcal F}_\omega$, $\omega = \omega_3 + \cdots$ dont la partie de degr\'e deux $[{\mathcal F}_{\omega_3}]$ est d\'efini par une $1-$forme ferm\'ee rationnelle. Comme on l'a dit au paragraphe \ref{section6}, ${\mathcal F}_\omega$ est aussi d\'efini par une $1-$forme ferm\'ee. Le cas le plus simple est celui o\`u $[{\mathcal F}_{\omega_3}]$ est un pinceau g\'en\'erique de coniques. Un tel pinceau est \`a conjugaison lin\'eaire pr\`es donn\'e par les niveaux de ${\displaystyle \frac{Q_1}{Q_2} = \frac{(x_0-x_1)(x_0+x_1)}{(x_0-x_2)(x_0+x_2)}}$. Le feuilletage associ\'e est donn\'e en coordonn\'ees homog\`enes par $$\omega_3 = Q_1dQ_2 - Q_2dQ_1. $$ Vu dans ${\mathbb P}^2_{\mathbb C}$ le feuilletage $[{\mathcal F}]_{\omega_3}$ a sept singularit\'es, dont quatre sont de type radial ($(1:1:1)$, $(1:1:-1)$, $(1:-1:1)$, $(-1:1:1)$) et trois sont des centres ($(1:0:0)$, $(0:1:0)$, $(0:0:1)$). Vu dans ${\mathbb C}^3$ le feuilletage ${\mathcal F}_{\omega_3}$ a donc sept lignes de singularit\'es dont quatre sont de type Kupka ($Q_1=Q_2=0$), et donc {\em persistantes} par perturbation. Soit $L = a_0x_0 + a_1x_1 + a_2x_2$, $a_i \in {\mathbb C}$, une forme lin\'eaire; introduisons la $1-$forme int\'egrable $\omega_L \in \Omega^1({\mathbb C}^3,0)$ $$\omega_L := Q_1dQ_2 - Q_2dQ_1 + Q_1Q_2dL. $$ On remarque de $\omega_L$ s'annule sur les quatre droites $Q_1=Q_2=0$ et que ${\mathcal F}_{\omega_L}$ a l'int\'egrale premi\`ere m\'e\-ro\-mor\-phe ${\displaystyle e^{-L}\frac{Q_1}{Q_2}}$. Consid\'erons l'\'eclatement $E: \widetilde{{\mathbb C}}^3 \rightarrow {\mathbb C}$ de l'origine; dans la premi\`ere carte $(x_0,t_1,t_2)$ o\`u $E = (x_0,t_1x_0,t_2x_0)$, on a: $$E^*\frac{\omega}{Q_1Q_2} = \frac{-2t_1 dt_1}{1-t_1^2} + \frac{2dt_2}{1-t_2^2} + (a_0 + a_1t_1 + a_2t_2)dx_0 + x_0(a_1dt_1 + a_2dt_2).$$ En particulier si $a_0\neq 0$, la singularit\'e initiale de $[{\mathcal F}_{\omega_3}]$ en $(1:0:0)$ a disparu et devient une tangence entre le feuilletage \'eclat\'e $E^{-1}({\mathcal F}_{\omega_L})$ et le diviseur exceptionnel $E^{-1}(0)$. Si $a_0a_1a_2 \neq 0$, les singularit\'es de $E^{-1}({\mathcal F}_{\omega_L})$ sont pr\'ecis\'ement sur les transform\'es strictes de $Q_1=Q_2=0$. Le long de ces quatre droites $D_1$, $D_2$, $D_3$, $D_4$ le feuilletage $E^{-1}({\mathcal F}_{\omega_L})$ est de type Kupka-radial. De sorte que si l'on \'eclate l'origine puis ces quatre droites par l'application $\sigma: M = \widetilde{{\mathbb C}}^3_{D_1,D_2,D_3,D_4} \rightarrow {\mathbb C}^3$, le feuilletage transform\'e $\sigma^{-1}({\mathcal F})$ est maintenant sans singularit\'es. On obtient ainsi en dimension trois un exemple de feuilletage {\em absolument dicritique} au sens de Cano-Corral \cite{Can-Cor}. On peut d'ailleurs g\'en\'eraliser cet exemple en prenant pour les $Q_i$ des polyn\^omes de degr\'e $k\geq2$ g\'en\'eriques.

Remarquons qu'en annulant certains $a_i$ on peut obtenir pour ${\mathcal F}_{\omega_L}$ quatre, cinq, six ou sept droites singuli\`eres. Ainsi sur cet exemple on peut \'eliminer de fa\c cons ind\'ependentes les droites singuli\`eres venant d'une configuration centrale.

En choisissant des entiers non nuls $k_0$, $k_1$, $k_2$ on peut aussi consid\'erer les feuilletages donn\'es par $\omega = Q_1dQ_2 - Q_2dQ_1 + Q_1Q_2 df $ avec $f = a_0x_0^{k_0} + a_1x_1^{k_1} + a_2x_2^{k_2}$. Par exemple dans $E^{-1}({\mathbb C}^3) = \widetilde{{\mathbb C}}^3$ le feuilletage relev\'e a au point $(1:0:0)$ une singularit\'e avec int\'egrale premi\`ere du type $xy + a_0z^{k_1}$. On en d\'eduit une infinit\'e non d\'enombrable de d\'eformations de ${\mathcal F}_{\omega_3}$ non conjugu\'ees.

On peut proc\'eder de m\^eme dans le cas de formes logarithmiques de type $(1,1,1,1)$ suivant la notation de \cite{Cer-Alc}. Un feuilletage g\'en\'erique de ce type est \`a conjugaison pr\`es de la forme: $$\omega_3 = x_0x_1x_2(x_0+x_1+x_2)\left( \lambda_0 \frac{dx_0}{x_0} + \lambda_1 \frac{dx_1}{x_1} + \lambda_2 \frac{dx_2}{x_2} - \frac{d(x_0+x_1+x_2)}{x_0+x_2+x_2} \right) $$ avec $\lambda_0 + \lambda_1 + \lambda_2 = 1$. Les points singuliers dans ${\mathbb P}^2_{\mathbb C}$ sont au nombre de sept; parmi eux six correspondent aux intersections des quatre droites invariantes. Ce sont des singularit\'es logarithmiques g\'en\'eriques pour des valeurs g\'en\'eriques des param\`etres $\lambda_k$. La septi\`eme singularit\'e est de type centre et se trouve au point $(\lambda_0:\lambda_1:\lambda_2)$. Vu dans ${\mathbb C}^3$ le feuilletage ${\mathcal F}_{\omega_3}$ a donc sept droites singuli\`eres. On peut comme pr\'ec\'edemment introduire $$\omega_L = \omega_3 + x_0x_1x_2 (x_0+x_1+x_2)dx_0. $$ Le feuilletage ${\mathcal F}_{\omega_L}$ a seulement six droites singuli\`eres. On peut aussi consid\'erer les $1-$formes \break $\omega_3 + x_0x_1x_2(x_0+x_1+x_2)df_k$ avec $f_k = x_0^k$ produisant une famille d\'enombrable de feuilletages avec m\^eme partie homog\`ene de degr\'e $3$ qui ne sont pas \'equivalents.

Terminons ce paragraphe par le cas exceptionnel. On consid\`ere sur ${\mathbb P}^3_{\mathbb C}$ le feuilletage ${\mathcal F}_\Gamma$ dont les feuilles sont les niveaux de la fonction rationnelle ${\displaystyle \frac{F^2}{G^3}}$ o\`u $F$ et $G$ sont donn\'es en coordonn\'ees homog\`enes par: ${\displaystyle F = x_3x_4^2 - x_1x_2x_4 + \frac{x_1^3}{3}}$, ${\displaystyle G = x_2x_4 - \frac{x_1^2}{2}}$. Il s'agit d'un feuilletage de degr\'e $2$, donn\'e par ${\displaystyle \Omega_3 = \frac{1}{x_4}(2GdF - 3FdG)}$, appel\'e {\em feuilletage exceptionnel}; il est exceptionnel \`a plusieurs titres. L'orbite de ${\mathcal F}_\Gamma$ sous l'action du groupe d'automorphismes Aut$({\mathbb P}^3_{\mathbb C})$, ou plut\^ot sa fermeture, est une composante irr\'eductible de l'espace des feuilletages de degr\'e $2$ sur ${\mathbb P}^3_{\mathbb C}$ \cite{Cer-Alc}. D'autre part ses feuilles sont des surfaces de degr\'e $6$, ce qui en un certain sens est un {\em degr\'e anormal} pour un feuilletage de degr\'e $2$. Son lieu singulier $\Gamma$ est constitu\'e d'une droite, d'une conique et d'une cubique gauche tangentes en un point $q$ \cite{Cer-Alc}. Soit $H \simeq {\mathbb P}^2_{\mathbb C}$ un hyperplan g\'en\'erique, i.e. transverse \`a $\Gamma$. La restriction de ${\mathcal F}_\Gamma$ \`a $H$ est un feuilletage de degr\'e $2$, qui a donc sept points singuliers compt\'es avec multiplicit\'e. Comme $\Gamma - \{q\}$ est constitu\'e des lignes de Kupka et $H$ est transverse \`a $\Gamma$, les six points de $H \cap \Gamma$ sont des points singuliers de ${\mathcal F}_\Gamma \big|_H$ qui ont pour nombre de Milnor $1$. Il y a donc un autre point $p \in H$ singulier pour ${\mathcal F}_\Gamma\big|_H$ avec nombre de Milnor \'egal a $1$. Comme $p$ n'appartient pas \`a $\Gamma$, le feuilletage ${\mathcal F}_\Gamma$ a une int\'egrale premi\`ere local $\varphi$ submersive en $p$, et la restriction $\varphi\big|_H$ est n\'ec\'essairement de Morse en $p$. Ainsi $p$ est un centre de ${\mathcal F}_\Gamma\big|_H$. Supposons $H$ donn\'e par $x_4 = ax_1 + bx_2 + cx_3 = L$, avec $abc \neq 0$, et notons $f$ et $g$ les restrictions de $F$ et $G$ \`a $H$: $f = F(x_1,x_2,x_3,L)$, $g = G(x_1,x_2,x_3,L)$. Le feuilletage ${\mathcal F}_\Gamma\big|_H$ est d\'efini par la forme: $$ \omega_3^H = \omega_3 = \frac{1}{L}(2gdf - 3fdg) \in \Omega^1({\mathbb C}^3,0).$$ Le feuilletage ${\mathcal F}_{\omega_3}$ a, pour $L$ g\'en\'erique, sept droites singuli\`eres dont six sont de type Kupka et la septi\`eme de type centre, en dehors de $0$.

On consid\`ere maintenant la perturbation $\omega = \omega_3 + \cdots$, de $3-$jet $\omega_3$, d\'efini par: $$\omega = \omega_3 + 2fgdL = \frac{1}{L} (2gdf - 3fdg + fgdL^2). $$ Cette $1-$forme est int\'egrable puisque: $$ \omega = \frac{fg}{L}\left(2 \frac{df}{f} - 3 \frac{dg}{g} + dL^2\right).$$ Notons que pour $H$ g\'en\'erique, le point $p \in H$ est en dehors des z\'eros de $F$ et $G$ et donc de $[f^{-1}(0)]$ et $[g^{-1}(0)]$. Si l'on proc\`ede \`a l'\'eclatement de $\omega$ par $E = (y_1,y_1y_2,y_1y_3)$ on obtient:
$$E^*\omega = \frac{y_1^5 \tilde{f} \tilde{g}}{y_1(a+by_2+cy_3)} \left(2\frac{d\tilde{f}}{\tilde{f}} - 3\frac{d\tilde{g}}{\tilde{g}} + dy_1^2(a+by_2+cy_3)^2 \right) $$ avec $f \circ E = y_1^3\tilde{f}$, $g \circ E = y_1^2 \tilde{g}$. On note que $\tilde{f}$ et $\tilde{g}$ sont des unit\'es en $p$ et comme ${\mathcal F}_\Gamma\big|_H$ a un centre en $p$, le $1$-jet de ${\displaystyle 2\frac{d\tilde{f}}{\tilde{f}} - 3\frac{d\tilde{g}}{\tilde{g}}}$ en $p$ est de type $udv + vdu$ pour un bon choix de coordonn\'ees $u,v$. Par suite le feuilletage $E^{-1}{\mathcal F}_\omega$ a en $p$ une singularit\'e isol\'ee; en r\'esulte que l'ensemble singulier de ${\mathcal F}_\omega$ est constitu\'e de six courbes lisses et non sept comme ${\mathcal F}_{\omega_3}$. Par suite ${\mathcal F}_\omega$ et ${\mathcal F}_{\omega_3}$ ne sont pas conjugu\'es.

On peut aussi consid\'erer les feuilletages ${\mathcal F}_{\omega^k}$ associ\'es aux $1-$formes: $$\omega^k = \omega_3 + fgkL^{k-1}dL,\quad k \in {\mathbb N}. $$ En \'etudiant la singularit\'e $E^{-1}{\mathcal F}_{\omega^k,p}$ on constate que les feuilletages ${\mathcal F}_{\omega^k}$ ne sont pas holomorphiquement conjugu\'es. Ceci indique en particulier que le feuilletage ${\mathcal F}_{\omega_3}$ n'est pas de d\'etermination finie; il est probable par contre que les ${\mathcal F}_{\omega^k}$ pour $k \geq1$ le soient.

\section{Conclusion}

Rappelons que la classification des feuilletages de degr\'e deux de ${\mathbb P}^2_{\mathbb C}$ n'ayant qu'une singularit\'e (donc \`a nombre de Milnor $7$) est faite dans \cite{Cer-Des}; tous sont d\'efinis par une $1-$forme ferm\'ee m\'eromorphe sauf un mod\`ele ayant une singularit\'e noeud-col, donc satisfaisant la condition de Kupka. Tous les r\'esultats qui pr\'ec\`edent sont r\'esum\'ees dans le Th\'eor\`eme \ref{main}, pr\'esent\'e dans l'introduction. On obtient en particulier la classification des feuilletages ${\mathcal F}_\omega$ sur ${\mathbb C}^3,0$ \`a partie initiale ${\rm In}(\omega)$ dicritique de degr\'e $3$, ${\rm Cod\ Sing\ In}(\omega)\geq2$ (modulo la condition de non r\'esonnance locale $b\notin {\mathbb Q}$).

Voici maintenant une liste de probl\`emes ouverts.

\begin{enumerate}[$1-$]

\item Traiter le cas r\'esonnant d'une singularit\'e de type $\Omega_2$ avec $b \in {\mathbb Q}$.

\item Soit $\omega_\nu$ une $1-$forme int\'egrable homog\`ene de degr\'e $\nu$, ${\rm Cod\ Sing}\ \omega_\nu \geq 2$; on consid\`ere l'ensemble $${\mathcal F}{\rm IN}(\omega_\nu):= \{{\rm feuilletages}\ {\mathcal F} = {\mathcal F}_\omega\ {\rm avec\ In}(\omega) = \omega_\nu\}.$$ Sur ${\mathcal F}{\rm IN}(\omega_\nu)$, il y a la relation d'\'equivalence de conjugaison holomorphe: ${\mathcal F}_\omega \sim {\mathcal F}_{\omega'}$ si et seulement si il existe un diff\'eomorphisme $\phi$ tangent \`a l'identit\'e tel que $\phi^*{\mathcal F}_\omega = {\mathcal F}_{\omega'}$. D\'ecrire l'espace de module ${\mathcal M}(\omega_\nu) = {\mathcal F}{\rm IN}(\omega_\nu)/\sim$; par exemple si $\omega_\nu$ satisfait le point 3 du Th\'eor\`eme \ref{main}, alors ${\mathcal M}(\omega_\nu)$ se r\'eduit \`a un point. Par contre nous venons de voir des exemples o\`u ${\mathcal M}(\omega_\nu)$ est infini ($\omega_\nu = Q_1dQ_2 - Q_2dQ_1$ comme ci-dessus).

\item G\'en\'eraliser les \'enonc\'es du Th\'eor\`eme \ref{main} aux dimensions sup\'erieures \`a $3$. La classification de \cite{Cer-Alc} permet de se ramener au cas o\`u la partie initiale ne d\'epend que de $3$ variables.

\item Donner une d\'emonstration g\'eom\'etrique, i.e. non calculatoire, du Th\'eor\`eme de Dulac.

\end{enumerate}


\end{document}